\documentclass[11pt,letterpaper]{amsart}

\usepackage[all]{xy}
\usepackage{amsmath}
\usepackage{amsfonts}
\usepackage{amssymb}
\usepackage{amsthm}
\usepackage{color}
\usepackage{enumerate}
\usepackage{hyperref}
\usepackage{fullpage}
\usepackage{graphicx}
\usepackage{url}

\usepackage{layouts}
\usepackage{float}

\usepackage[english]{babel}

\title{Bibliography management: BibTeX}

\usepackage{amsmath,amssymb}

\newcommand{\xrightarrowdbl}[2][]{%
  \xrightarrow[#1]{#2}\mathrel{\mkern-14mu}\rightarrow
}

\usepackage{amsmath}

\usepackage{amsmath}

\theoremstyle{definition}
\newtheorem{defn}{Definition}[section]

\newtheorem{rmk}[defn]{Remark}

\theoremstyle{plain}
\newtheorem{thm}[defn]{Theorem}

\newtheorem{lem}[defn]{Lemma}
\newtheorem{prop}[defn]{Proposition}
\newtheorem{cor}[defn]{Corollary}

\newtheorem{exmp}[defn]{Example}

\usepackage{graphicx, caption}

\hypersetup{
    colorlinks=true,
    linkcolor=green!33!blue,
   filecolor=magenta,      
    urlcolor=green!50!black,
    citecolor=blue!55!black
}
 
\usepackage{amsbsy}
\usepackage{amstext}
\usepackage{amsxtra}
\usepackage{amsbsy}
\usepackage{amscd}

\usepackage{fancyhdr}

\usepackage{url}
\usepackage{hyperref}
\usepackage{graphicx}
\usepackage[numbers]{natbib}
\usepackage{amsmath}
\usepackage{mathtools}
\usepackage{xcolor}
\usepackage{fancyhdr}
\usepackage{blindtext}

\usepackage{a4wide}

\usepackage{array, booktabs, makecell}
\usepackage{siunitx}

\usepackage[utf8x]{inputenc}
\usepackage{amsmath}
\usepackage{tikz}

\usepackage{pict2e}

\usepackage{tikz}
\usepackage{calc}

\usepackage{pict2e}

\usepackage[normalem]{ulem} 
\usepackage{soul} 

\usepackage[normalem]{ulem}

\usepackage{amsmath}
\usepackage{tikz}
\usepackage{calc}
\usepackage{tikz-cd}

\usepackage{graphicx} 
\usepackage{caption}
\usepackage{mwe}

\usepackage{dirtytalk}

\usetikzlibrary{calc} 
\usepackage{tikz-cd}
\usepackage{amsmath}

 \usepackage{graphicx}
\graphicspath{ {images/} }

\usepackage{graphicx}

\usepackage{graphicx}

\usepackage{animate}

\usepackage{pgf,tikz}\usepackage{mathrsfs}\usetikzlibrary{arrows}

\usepackage{lipsum}

\newcommand\blfootnote[1]{%
  \begingroup
  \renewcommand\thefootnote{}\footnote{#1}%
  \addtocounter{footnote}{-1}%
  \endgroup
}

\def\C{\ensuremath{\mathbb{C}}}

\def\H{\ensuremath{\mathbb{H}}}

\def\L{\ensuremath{\mathbb{L}}}

\def\P{\ensuremath{\mathbb{P}}}

\def\R{\ensuremath{\mathbb{R}}}

\def\FF{\ensuremath{\mathcal F}}

\def\HH{\ensuremath{\mathcal H}}
\def\II{\ensuremath{\mathcal I}}

\def\NN{\ensuremath{\mathcal N}}
\def\OO{\ensuremath{\mathcal O}}

\def\TT{\ensuremath{\mathcal T}}
\def\UU{\ensuremath{\mathcal U}}

\def\Stab{\mathop{\mathrm{Stab}}\nolimits}
\def\tilt{\mathop{\mathrm{tilt}}\nolimits}

\def\ch{\mathop{\mathrm{ch}}\nolimits}

\def\dim{\mathop{\mathrm{dim}}\nolimits}
\def\Supp{\mathop{\mathrm{Supp}}\nolimits}

\def\Ext{\mathop{\mathrm{Ext}}\nolimits}
\def\Hom{\mathop{\mathrm{Hom}}\nolimits}

\def\Hom{\mathop{\mathrm{Hom}}\nolimits}

\def\im{\mathop{\mathrm{im}}\nolimits}
\def\Imm{\mathop{\mathrm{Im}}\nolimits}

\def\tilt{\mathop{\mathrm{tilt}}}

\def\Stab{\mathop{\mathrm{Stab}}\nolimits}

\def\into{\ensuremath{\hookrightarrow}}
\def\onto{\ensuremath{\twoheadrightarrow}}

\begin{document}

\title[geometry of canonical genus four curves]{\texorpdfstring{geometry of canonical genus four curves}{geometry of canonical genus four curves}}

 \author{Fatemeh Rezaee}

\address{Centre for Mathematical Sciences, University of Cambridge, Wilberforce Road, CB3 0WA, Cambridge, United Kingdom}
\curraddr{ETH Z\"urich}
\email{fr414@cam.ac.uk}



\begin{abstract}

We apply the machinery of Bridgeland stability conditions  on derived categories of coherent sheaves to describe the geometry of classical moduli spaces associated with canonical genus four space curves via an effective control over its wall-crossing. 
This article provides the first description of a moduli space of Pandharipande-Thomas stable pairs which is used as an intermediate step towards the description of the associated Hilbert scheme, which in turn is the first example where the components of a classical moduli space were completely determined via wall-crossing. 

 We give a full list of irreducible components of the space of stable pairs, along with a birational description of each component, and a partial list for the Hilbert scheme. 

There are several long standing open problems regarding classical sheaf theoretic moduli spaces, and the present work will shed light on further studies of such moduli spaces such as Hilbert schemes of curves and moduli of stable pairs which are very hard to tackle without the wall-crossing techniques.

\end{abstract}

\maketitle

    \tableofcontents

\section{Introduction}\blfootnote{Mathematics Subject Classification codes: 	14C05, 13D09, 	14H60}
Despite their very natural definition as the parametrization spaces of sub-varieties in projective space,  Hilbert schemes are among the most badly behaved moduli spaces in algebraic geometry; for example  
Murphy's Law holds for these spaces (
\cite{Vakil06}), or see \cite{Jelisiejew20} for pathologies of the Hilbert scheme of points.

For $\mathbb{P}^{2}$ as an ambient space, the Hilbert scheme of points is well-studied, but already in the case of $\mathbb{P}^{3}$, few general results are known, and not many examples are fully understood. Some classical results are known about some Hilbert scheme of curves in $\P^3$: For instance, see \cite{PS85} and \cite{EPS87} for the Hilbert scheme of twisted cubics, and  \cite {AV92}, \cite{Go08} and \cite{CN12} for the Hilbert scheme of elliptic quartics. Both Hilbert schemes have two irreducible components. Also see \cite{Nollet97} and \cite{Nollet-Schlesinger2003} for the Hilbert schemes of Cohen-Macaulay curves, and \cite{Henni-Jardim2018} for the Hilbert schemes of points.  The first examples of wall-crossing for Hilbert scheme of curves in $\mathbb{P}^{3}$ can be found in \cite{Sch, Xia} for twisted cubics, and in \cite{GHS} for elliptic quartics.

Our goal is to describe the geometry of the Hilbert scheme $\mathcal{H}ilb^{6t-3}(\mathbb{P}^{3})$ containing canonical genus four curves, as well as the associated moduli space of PT stable pairs. 

The space $\mathcal{H}ilb^{6t-3}(\mathbb{P}^{3})$ parametrizes more objects than only the canonical genus four pure curves; 
considering the formula $p(t)=dt+1-g$ for the Hilbert polynomial of genus $g$ and degree $d$ in $\mathbb{P}^{3}$, 
some curves of higher genus together with some extra points can have the same Hilbert polynomial.

A \textit{PT stable pair} was defined by Pandharipande and Thomas to be a pair $(\FF,s)$, where $\FF$ is a  1-dimensional pure sheaf on $X$ and  $s\colon\mathcal{O}_{X} \rightarrow \mathcal{F}$ a section with zero-dimensional cokernel (Definition \ref{PTdef}). 
We will prove the following two main theorems describing the space of PT stable pairs and the Hilbert scheme associated to the canonical curves of genus 4:

\begin{thm}  [See Theorem \ref{moduliPT1}] \label{moduliPT}
The moduli space of PT stable pairs $\mathcal{O}_{\mathbb{P}^{3}} \rightarrow \mathcal{F}$, where $Ch(\mathcal{F})=(0,0,6,-15)$ consists of components birational to the following eight irreducible components (the first one is 24-dimensional, the last one is 36-dimensional, and the rest are all 28-dimensional):

(1) A $\mathbb{P}^{15}$-bundle over $|\OO_{\mathbb{P}^3}(2)|$, which generically parametrizes (2,3)-complete intersections,

(2) a $\mathbb{P}^{17}$-bundle over $\mathbb{G}r(2,4) \times \mathfrak{Fl}_2 $, which generically parametrizes the union of a line and a plane quintic intersecting the line together with a choice of two points on the quintic,

(3) a $\mathbb{P}^{18}$-bundle over $\mathcal{U}  \times \mathfrak{Fl}_1 $,  which generically parametrizes the union of  a  line in $\P^3$ together with a choice of a point on it,  and  a plane quintic intersecting the line, together with a choice of a point  on it,

(4) a $\mathbb{P}^{19}$-bundle over $\mathbb{G}r(2,4) \times \mathfrak{Fl}_1 $, which  generically parametrizes the disjoint union of a  line in $\P^3$ and a plane quintic together with a choice of a point  on it,

(5) a $\mathbb{P}^{19}$-bundle over $(\mathcal{U} \times_{\mathbb{G}r(2,4)}\mathcal{U}) \times (\mathbb{P}^{3})^{\vee} $, which  generically parametrizes the union of a line in $\P^3$ together with a choice of two points on it,  and a plane quintic intersecting the line,

(6) a $\mathbb{P}^{20}$-bundle over $\mathcal{U}  \times (\mathbb{P}^{3})^{\vee} $, which  generically parametrizes the disjoint union of  a line in $\P^3$ together with a choice of a point on it, and a plane quintic,

(7) a $\mathbb{P}^{21}$-bundle over $\mathfrak{Fl}_2,$ which generically parametrizes the union of a plane quartic with a thickening of a line in the plane, and

(8) a $\mathbb{P}^{21}$-bundle over $\mathfrak{Fl}_6,$ which generically parametrizes a plane degree 6 curve together with a choice of 6 points on it.

\noindent
Here  $|\OO_{\mathbb{P}^3}(2)|$ parametrizes quadric surfaces in $\mathbb{P}^{3}$, $\mathcal{U}$ is the universal line over $\mathbb{G}r(2,4)$, and $\mathfrak{Fl}_{l} $ is the space  parametrizing flags $Z_{l} \subset P \subset \mathbb{P}^{3}$
where $P$ is a plane and $Z_{l}$ a zero-dimensional subscheme of length $l$.
\end{thm}

For the Hilbert scheme we obtain:

\begin{thm} [See Theorem \ref{HScomp1}] \label{HScomp} 
The Hilbert scheme ${\mathcal{H}ilb^{6t-3}(\mathbb{P}^{3})}$ has  components 
birational to:

1)  The main component, $\mathscr{H}_{CM}$, which is a $\mathbb{P}^{15}$-bundle over ${|\OO_{\mathbb{P}^3}(2)|}$ (24-dimensional),

2) $\mathscr{H}'_{CM}$ which generically parametrizes the union of a plane quartic with a thickening of a line in the plane (28-dimensional),

3) $\mathscr{H}_1$ which generically parametrizes the disjoint union of a line in $\P^3$ and a plane quintic together with 1 floating point (30-dimensional),

4) $\mathscr{H}_2$ which generically parametrizes the union of a line in $\P^3$ and a plane quintic together with 2 floating points (32-dimensional), and

5) $\mathscr{H}_6$ which generically parametrizes a plane sextic together with 6 floating points.
\\\\
The first four components are irreducible.

\end{thm}

As a remark, we note that we expect  $\mathscr{H}_6$ to be irreducible of dimension 48.

Notice that a smooth non-hyperelliptic genus 4 curve $C$ embeds into $\mathbb{P}^{3}$ as a (2,3)-complete intersection curve. 
The question is how to compactify this 24-dimensional space of such curves. As we have an embedding in the projective space, a classical answer would be considering the Hilbert scheme of such curves. But this has many irreducible components;  for example  plane sextics with 6 floating points added yields a component of dimension $48$. It would be hard to guess and list the rest of the irreducible components.

Therefore, instead of studying the Hilbert scheme directly, we consider Bridgeland stability conditions on $\mathrm{D}^{b}(\mathbb{P}^{3})$. Depending on a choice of a stability condition $\sigma \in \Stab(\mathbb{P}^{3})$, we get $\mathcal{M}_\sigma(1, 0, -6, 15)$, the moduli space of $\sigma$-stable complexes $E$ with $\mathrm{ch}(E)=\mathrm{ch}(\mathcal{I}_{C})$.
Following a path in the space of stability conditions, we want to understand how $\mathcal{M}_\sigma(1, 0, -6, 15)$ changes via wall-crossing. 
We studied the beginning of this path in \cite{R1}. For example, the first chamber gives a $\P^{15}$-bundle over $\P^9$ where $\P^9$ corresponds to  a choice of a quadric, and $\P^{15}$ corresponds to a choice cubic modulo the ideal generated by the quadric. This is evidently a very efficient compactification of the space of canonical genus 4 curves, which are (2,3)-complete intersections.  On the other hand in the large-volume limit, via rough control over wall-crossing, we recover the geometry of the Hilbert scheme, for which there are no prior results. 


\subsection*{Strategy of the proof}
There are two non-trivial extremes on a path we are taking from the empty space to the large volume limit: one is an efficient compactification of the (2,3)-complete intersection curves and the other is the Hilbert scheme. The wall after the blow-up of the compactification which introduces a novel birational phenomenon is described in \cite{R1}.  The strategy to reach the Hilbert scheme uses the space of PT-stable pairs as an intermediate step (see Figure \ref{walls0}). For each wall $\langle A, B \rangle$ on the right side of of the left branch of the hyperbola $\Imm (Z_{\alpha, \beta, s})=0$ (which will be denoted by $\H$) as in Theorem \ref{main}, the newly stable objects are given by an exact triangle $A \to E \to B$. For each $A, B$ this gives a locus $\P(\Ext^1(B,A))$ inside the moduli space after crossing the wall. We stratify the space of pairs $(A,B)$ by $\dim(\P(\Ext^1(B,A)))$; for each stratum, we describe a general element and decide whether it is in the closure of other strata. This will describe the components of the moduli space of stable pairs (see Theorem \ref{moduliPT1}). To recover the Hilbert scheme from this moduli space, we analyze the DT/PT wall (i.e., the wall between the two adjacent chambers associated with the Hilbert scheme and the space of stable pairs). We investigate which components in the space of stable pairs survive in the Hilbert scheme, and what new components are created (see Theorem \ref{HScomp1}).

\begin{figure}[htp] 
\centering
\includegraphics[width=12cm]{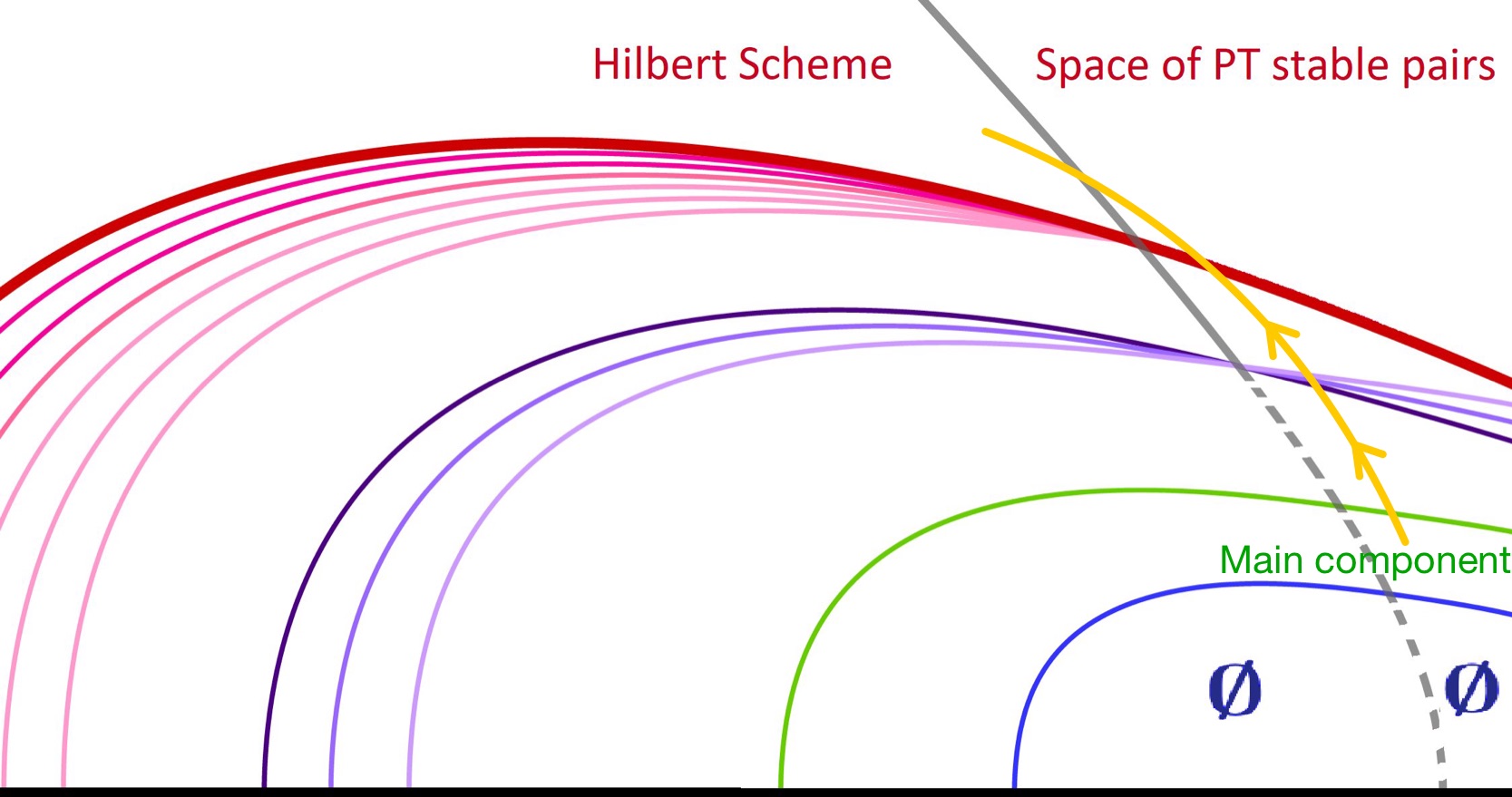}
\caption{we cross the walls along the yellow path which is close enough to the hyperbola}
\label{walls0}
\end{figure}

\subsection*{Acknowledgements}
First and foremost, I would like to thank Arend Bayer for suggesting the problem, generosity of time, invaluable comments on preliminary versions, and helpful discussions. Special thanks to Benjamin Schmidt for the great suggestion to work on the stable pairs side, and for helpful discussions. Special thanks to Qingyuan Jiang for discussions on the irreducibility of projectivization loci. Special thanks to Dominic Joyce for the last-minute comments.  This work benefited from useful discussions with Antony Maciocia. I am grateful for comments by Aaron Bertram, Tom Bridgeland, Ivan Cheltsov, Dawei Chen,  Daniel Huybrechts, Marcos Jardim, Dominic Joyce, Emanuele Macr\`{i}, Scott Nollet,  Balázs Szendröi, Richard Thomas, Yukinobu Toda, and Bingyu Xia. I am also grateful for careful reading, and several helpful comments and suggestions by the anonymous referee. The author was supported by PCDS scholarship, the school of mathematics of the University of Edinburgh scholarship, ERC Starting grant WallXBirGeom, no. 337039, and ERC Consolidator grant WallCrossAG, no. 819864.  
This material is partially based upon work supported by the NSF under Grant No. DMS-1440140 while the author was in residence at the MSRI in Berkeley, California, during the spring 2019 semester. The author was also supported by the EPSRC grant EP/T015896/1, during final edits. The revisions for journal publication were done while the author was supported by UKRI grant No. EP/X032779/1 (in lieu of an MSCA grant). The pictures were generated with GeoGebra, and some computations were checked using Macaulay2 (\cite{GS}).


\subsection*{Notation}

\begin{center}
   \begin{tabular}{ r l }
    $\langle A,B \rangle$: & the wall  where strictly semistable objects have Jordan-H\"older factors of type \\ & $A$ and  $B$\\
      $\Stab(\mathbb{P}^{3})$: & the space of Bridgeland stability conditions in $\mathbb{P}^{3}$ \\
     
      $\Stab^{\tilt}(\mathbb{P}^{3})$: & the space of tilt-stability conditions in $\mathbb{P}^{3}$ \\
        $\lambda_{\alpha, \beta, s}$: & Bridgeland stability slope (Definition \ref{defbridgeland})\\
    
      $\nu_{\alpha, \beta}$: & tilt-stability slope (Definition \ref{deftilt})  \\
        $\H$: & (the left branch of) the hyperbola defined by 
      $\Imm(Z_{\alpha, \beta, s}(1,0,-6,15))=0$\\

      $\HH^{i}$: & the $i-th$ cohomology object in the corresponding heart\\
         $\mathrm{H}^{i}$: & the $i-th$ sheaf cohomology  group\\
      $\otimes$: & derived tensor (unless otherwise explicitly stated)\\
        $p_a$: & arithmetic genus\\

       $\mathfrak{Fl}_i$: &  the space parametrizing flags $Z_{i} \subset P \subset \mathbb{P}^{3}$ where $P$ is a plane and $Z_{i}$ a \\& zero-dimensional subscheme of length $i$

 \end{tabular}

        \begin{tabular}{ r l }     \\
      $\mathfrak{PL}_i$: &  the space parametrizing a plane, a line, and $i$ floating points
     \\
             $\mathcal{U}$:&  the universal line over $\mathbb{G}r(2,4)$\\

   \end{tabular}
\end{center}
$$$$
\subsection*{Conventions}

When there is no confusion, the subobject and the quotient of the defining short exact sequence of any wall will be denoted by $A$ and $B$, respectively. Notice that we cross the walls towards the large volume limit. Note that by \textit{large volume limit}, we mean the region in $\mathrm{Stab}(\mathbb{P}^3)$ where $\alpha \to \infty$. When we say \textit{near hyperbola}, we mean a sufficiently small tubular neighborhood of the locus of the hyperbola in the plane, so that there is no other wall intersects that region.  Let $\tilde{C} \subset \mathbb{P}^3$ be a 1-dimensional subscheme. If $\tilde{C}$ is not pure dimensional, one can consider its maximal 0-dimensional subsheaf $\mathcal{T} \subset \mathcal{O}_{\tilde{C}}$ so that $ \mathcal{O}_{\tilde{C}}/\mathcal{T}=\mathcal{O}_C$, for some curve $C \subset \tilde{C}$ as the maximal subscheme of $\tilde{C}$ with pure dimension 1.
Now, if a point $p \in \Supp(\mathcal{T}) \cap C$ , then it is called an \textit{embedded points}, while
if $p \in \Supp(\mathcal{T})$, but $p \not\in C$, then $p$ is called a \textit{floating points}.

\section{Bridgeland stability conditions and stable pairs } \label{Section2}
In this section, we review Bridgeland stability conditions on $\mathbb{P}^{3}$, and the notion of PT stable pairs. 

\subsection{Bridgeland stability conditions on $\mathbb{P}^{3}$}

In this subsection we briefly define stability conditions on the bounded derived category $\mathrm{D}^b(\mathbb{P}^{3})$ of coherent sheaves on $\mathbb{P}^{3}$, following the construction in ~\cite{BMT}.  Let $\mathrm{Coh}(\mathbb{P}^{3})$ be the abelian category of coherent sheaves (as an initial heart of a bounded t-structure with usual notion of torsion pair) on  $\mathbb{P}^{3}$.  Let $\alpha>0, \beta \in \mathbb{R}$ be two real numbers. We define the twisted Chern characters $\mathrm{ch}^{\beta}(E)= e^{-\beta H}.\mathrm{ch}(E)$, where $\mathrm{H}$ denotes the hyperplane class. For $E \in \mathrm{D}^{b}(\mathbb{P}^{3})$, we define the \textit{twisted slope function} by $\mu_{\beta}(E):=  \frac{ch^{\beta}_{1}(E)}{\mathrm{ch}_{0}(E)}$  if $\mathrm{ch}_{0}(E) \neq 0$, and $\mu_{\beta}(E)=+ \infty$ otherwise. 

\begin{defn} \label{deftilt}
        By \textit{tilting}, one can define a new heart of a bounded t-structure as follows: the torsion pair is defined by 
       \begin{equation*}
\begin{aligned}\label{(2.1)}
\mathcal{T}_{\beta}:= \text{{\{$E \in \mathrm{Coh}(\mathbb{P}^{3})\colon \mu_{\beta}(G)>0$  for all $E  \twoheadrightarrow G $}\}},   \\
\mathcal{F}_{\beta}:= \text{{\{$E \in \mathrm{Coh}(\mathbb{P}^{3})\colon \mu_{\beta}(F) \leq 0$ for all $F  \hookrightarrow E$}\}}.
\end{aligned}
\end{equation*}
The new \textit{heart of a bounded t-structure} can be defined as $\mathrm{Coh}^{\beta}(\mathbb{P}^{3}):= \langle \mathcal{F}_{\beta}[1], \mathcal{T}_{\beta} \rangle$.

Now we are ready to define the notion of tilt-stability:

 The \textit{central charge} and the corresponding \textit{slope function} for the new heart can be defined as

\begin{equation*}
Z^{\tilt}_{\alpha, \beta}:=-(\mathrm{ch}_{2}-\beta \mathrm{ch}_{1}+(\beta^{2}/2)\mathrm{ch}_{0})+(\alpha^{2}/2)\mathrm{ch}_{0}+i(\mathrm{ch}_{1}-\beta \mathrm{ch}_{0})=-(\mathrm{ch}^{\beta}_{2})+(\alpha^{2}/2)\mathrm{ch}^{\beta}_{0}+i(\mathrm{ch}^{\beta}_{1}), 
\end{equation*}

and 
\begin{equation*}
\nu_{\alpha, \beta}:=-\frac{\mathrm{Re}(Z^{\tilt}_{\alpha, \beta})}{\mathrm{Im}(Z^{\tilt}_{\alpha, \beta})}=\frac{\mathrm{ch}^{\beta}_{2}-(\alpha^{2}/2)\mathrm{ch}^{\beta}_{0}}{\mathrm{ch}^{\beta}_{1}},
\end{equation*}

with $\nu_{\alpha, \beta}(E)=+\infty$ if $\mathrm{ch}^{\beta}_{1}(E)=0$. The pair $(\mathrm{Coh}^{\beta}(\mathbb{P}^{3}),Z^{\tilt}_{\alpha, \beta})$ is a weak stability condition also known as \textit{tilt-stability condition}.
\end{defn}

We denote by $\Stab^{\tilt}(\mathbb{P}^{3})$, the space of all tilt-stability conditions. It was conjectured in ~\cite{BMT} for arbitrary threefolds, and proved in by Macr\`{i} in \cite{M} for $\mathbb{P}^{3}$ that tilting again gives a Bridgeland stability condition as follows:

\begin{defn} \label{defbridgeland}
We define a \textit{torsion pair} similarly as for the tilting case:

 \begin{equation*}
\begin{aligned}\label{(2.2)}
\mathcal{T}_{\alpha,\beta}:=\text{{\{$E \in \mathrm{Coh}^{\beta}(\mathbb{P}^{3}): \nu_{\alpha, \beta}(G)>0$ for all $E  \twoheadrightarrow G$ }\}},   \\
\mathcal{F}_{\alpha,\beta}:=\text{{\{$E \in \mathrm{Coh}^{\beta}(\mathbb{P}^{3}): \nu_{\alpha, \beta}(F) \leq 0$ for all $F  \hookrightarrow E$}\}}.
\end{aligned}
\end{equation*}

Now define a new \textit{heart, central charge, and slope} respectively as follows:

\begin{equation*}\mathrm{Coh}^{ \alpha, \beta}(\mathbb{P}^{3}) :=\langle \mathcal{F}_{\alpha,\beta}[1], \mathcal{T}_{\alpha,\beta}\rangle,
\end{equation*}

\begin{equation*}
Z_{\alpha,\beta,s}:=-\mathrm{ch}^{\beta}_{3}+(s+1/6)\alpha^{2}\mathrm{ch}^{\beta}_{1}+i(\mathrm{ch}^{\beta}_{2}-\alpha^{2}/2\mathrm{ch}^{\beta}_{0}),
\end{equation*}
and (for $s>0$)
\begin{equation*}
\lambda_{\alpha, \beta, s}:= -\frac{\mathrm{R}e(Z_{\alpha,\beta,s})}{\mathrm{I}m(Z_{\alpha,\beta,s})}.
\end{equation*}
with $\lambda_{\alpha, \beta, s}(E)=+\infty$ if $\mathrm{Im}(Z_{\alpha,\beta,s})(E)=0$.
\end{defn}

The pair $\sigma_{\alpha, \beta, s}=(\mathrm{Coh}^{\alpha, \beta}(\mathbb{P}^{3}), Z_{\alpha,\beta,s})$ (when exists) is called \textit{Bridgeland stability condition}.
   
   Before going further, we have a formal definition of a wall and chamber:
   
   \begin{defn} [{\cite [Definition 5.26]{MS2},\cite[page 161]{MS18}}]
   
  Let $\Lambda$ be a finite rank lattice, and $v,v' \in \Lambda^{*}$ two non-parallel classes. We fix the parameter $s$ in the definition of $\lambda_{\alpha, \beta, s}$, and define a \textit{numerical wall} in Bridgeland stability with respect to 
a class $v$ to be a non-trivial proper subset of the stability space which is defined as
$$\mathcal{W}_{v,v'}=\{(\alpha, \beta) \in\R_{>0} \times \R: \lambda_{\alpha, \beta, s}(v)=\lambda_{\alpha, \beta, s}(v')\}.$$

 An \textit{actual wall} is the subset $\mathcal{W}'$ of a numerical wall consisting of those points for which there exists properly semistable objects with class $v$ (we can give a similar definition for tilt-stability).

A \textit{chamber} is defined as a connected component of the complement of the set of actual walls in $\R_{>0} \times \R$.

   \end{defn}
   
The main point to show that $(\mathrm{Coh}^{\alpha, \beta}(\mathbb{P}^{3}), Z_{\alpha,\beta,s})$ defines a Bridgeland stability condition (for all $s>0$) is a Bogomolov-type inequality, which we refer to it as \textit{BMT inequality} (Theorem \ref{BMT}). Before stating that, we have the classical Bogomolov-Gieseker inequality:
\begin{thm}  [{\cite[Corollary 7.3.2]{BMT}}] \label{BG}
 Any $\nu_{\alpha, \beta}$-semistable object $E \in \mathrm{Coh}^{\beta}(\mathbb{P}^{3})$ satisfies
 $$2\mathrm{ch}_{0}(E)\mathrm{ch}_{2}(E) \leq \mathrm{ch}^{2}_{1}(E).$$
\end{thm}

\begin{thm}
 [{\cite[Lemma 8.8]{BMS}, \cite[Theorem 1.1]{M}}] \label{BMT}
 
 Any $\nu_{\alpha, \beta}$-semistable object $E \in \mathrm{Coh}^{\beta}(\mathbb{P}^{3})$ satisfies
\begin{align*}
    \alpha^{2} [(\mathrm{ch}^{\beta}_{1} (E))^{2} - 2(\mathrm{ch}^{\beta}_{0} (E) (\mathrm{ch}^{\beta}_{2} (E))]+4(\mathrm{ch}^{\beta}_{2} (E))^{2} - 6(\mathrm{ch}^{\beta}_{1} (E)) \mathrm{ch}^{\beta}_{3} (E)\geq0,
\end{align*}
 and $(\mathrm{Coh}^{\alpha, \beta}(X), Z_{\alpha,\beta,s})$ is a Bridgeland stability condition for all $s > 0$. Moreover, the support property holds for $(\mathrm{Coh}^{\alpha, \beta}(X), Z_{\alpha,\beta,s})Y$.
\end{thm}

The support property implies the manifold of all (Bridgeland) stability conditions $\Stab(\mathbb{P}^{3})$ admits a chamber decomposition, depending on $v$, such that
(i) for a chamber $C$, the moduli space $\mathcal{M}_{\sigma}(v) = \mathcal{M}_{C}(v) $ is independent of the choice of
$\sigma \in C$, and
(ii) walls consist of stability conditions with strictly semistable objects of class $v$ (\cite{BM}).

There is also a well-behaved wall-chamber structure in $\Stab^{\tilt}(\mathbb{P}^{3})$. The last part of the following Theorem was proven for surfaces in \cite{MacA}:
\begin{thm}[{\cite[Proposition B.2]{BMS}}] \label{WC}
The function $\mathbb{R}_{> 0} \times \mathbb{R} \rightarrow \Stab^{\tilt}(\mathbb{P}^{3})$ defined
by $(\alpha, \beta) \rightarrow (\mathrm{Coh}^{\beta}(X), Z_{\alpha,\beta})$ is continuous. Moreover, walls with respect to a class $v$ in the image of this map are locally finite. In addition, the walls in the tilt-stability space are either nested semicircles or vertical lines.
\end{thm}

\begin{rmk} \label{remJH}
Note that the Jordan-H\"older factors of an object on a wall are stable along the entire wall.
\end{rmk}

For more details on Bridgeland stability conditions on $\mathbb{P}^{3}$, we refer to \cite{Sch}.

\subsection{Pandharipande-Thomas stable pairs} In order to make the floating/embedded points for the curves parametrizing the Hilbert scheme of curves more under control, Pandharipande and Thomas introduced the following notion: 
     

\begin{defn}[{\cite[page 408]{PT}}] \label{PTdef}
A \textit{stable pair} $(\FF,s)$ consists of a  1-dimensional pure sheaf $\mathcal{F}$ on $X$ and a sections $s\colon\mathcal{O}_{X} \rightarrow \mathcal{F}$ with zero-dimensional cokernel.
\end{defn}  

 Equivalently, we have the following characterization of stable pairs:
 \begin{lem}[{\cite[page 408]{PT}}] \label{charPT}
 An object $E \in \mathrm{D}^{b}(X)$ is a stable pair if: i) $\HH^0(E)$ is an ideal sheaf of a Cohen-Macaulay curve, ii) $\mathcal{H}^1(E)$ is a sheaf with 0-dimensional support, and iii) $\mathrm{Hom}(\OO_p[-1],E)=0$ for all $p \in X$.
 \end{lem}

We refer to the zero-dimensional cokernel as \textit{\say{points on the curve}}.

\section{Walls} \label{Walls}
According to Theorem \ref{WC}, there is  a wall and chamber structure in the stability manifold. In this section, we numerically describe the walls in $\Stab^{\tilt}(\mathbb{P}^{3})$ with respect to $\mathrm{ch}(\mathcal{I}_C)$, where $C$ is a canonical genus four curve, and give a geometric description of the walls.

Note that for a canonical genus 4 curve $C$ in $\mathbb{P}^{3}$, we have $\mathrm{ch}(\mathcal{I}_{C})=(1,0,-6,15)$ ({\cite[Proposition 3.1]{R1}}).

\begin{rmk}
 We notice that the Hilbert polynomial of a canonical genus four curve is $6t-3$: We denote the corresponding Hilbert scheme by $\mathcal{H}ilb^{6t-3}(\mathbb{P}^{3})$.
\end{rmk}

Consider the hyperbola $\H$ in the $(\alpha, \beta)$-plane defined by $\mathrm{Im}(Z_{\alpha, \beta, s}(v))=0$. For such $\alpha, \beta, s$, semistable objects of Chern character $\nu$ have phase $0$. Moreover, these semistable objects have positive and negative phases with respect to the stability conditions on the left and right side of the hyperbola, respectively. Thus on the left we work with  $\mathrm{Coh}^{\beta}(\mathbb{P}^3)$ and $\mathrm{Coh}^{\alpha,\beta}(\mathbb{P}^3)$, for tilt and Bridgeland stability, respectively, and  on the right side of the hyperbola we work in $\mathrm{Coh}^{\beta}(\mathbb{P}^3)[-1]$ and $\mathrm{Coh}^{\alpha,\beta}(\mathbb{P}^3)[-1]$.

Theorem \ref{WC} gives an order for the walls or semicircles in $\Stab^{\tilt} (\mathbb{P}^{3})$. We refer to the semicircle with the smallest radius as the first wall, and so on.

First, we recall a number of results from \cite{R1}:

\begin{lem}[{\cite[Lemma 3.2]{R1}, \cite[Lemma 5.4]{Sch}}]  \label{lem:lem_7}   

Let $\beta$ be an integer, and $E$ a tilt semistable object in $\mathrm{Coh}^{\beta}(\mathbb{P}^{3})$.

1. If $\mathrm{ch}^{\beta}(E)=(1,1,d,e)$, then $d-1/2 \in \mathbb{Z}_{\leq 0}$. If $d=1/2$, then $E \cong \mathcal{I}_{Z}(\beta +1)$ for a zero dimensional subscheme $Z$ in $\mathbb{P}^{3}$ of length $1/6-e$. If $d=1/2-D$ where $D=1,2$, then we have $E \cong \mathcal{I}_{C_D}(\beta +1)$ where $C_D$ is a rational degree $D$ curve, plus $D-e-5/6$ floating or embedded points in $\mathbb{P}^{3}$.

2. If $\mathrm{ch}^{\beta}(E)=(0,1,d,e)$, then $d+1/2 \in \mathbb{Z}$ and $E \cong \mathcal{I}_{Z/P}(\beta +d +1/2)$ in which $Z$ is a zero-dimensional subscheme supported in a plane $P$ in $\mathbb{P}^{3}$, and of length $1/24+d^{2}/2-e$.

\end{lem}

\begin{lem}[{\cite[Lemma 3.4]{R1}}] \label{(0, >)}
 Let $E \in \mathrm{Coh}^{\beta}(\mathbb{P}^{3})$ be a $\nu_{\alpha, \beta}$-semistable object with $\mathrm{ch}(E)=(0,2,-8,e)$. Then $e \leq 49/3$. Moreover, if the equality holds, then $E=  \mathcal{O}_{ Q}(-3)$ for a (possibly singular) quadric surface $Q$ in $\mathbb{P}^{3}$.
\end{lem}

To describe the walls in the space of Bridgeland stability conditions, we need the following result relating tilt-stability and Bridgeland stability:

\begin{lem}[{\cite[Lemma 8.9]{BMS}}] \label{BMS8.9}
 Let $E \in \mathrm{Coh}^{ \alpha, \beta}(\mathbb{P}^{3})$ be a $\lambda_{\alpha, \beta, s}$-semistable object, for all $s\gg1$ sufficiently big. Then it satisfies one of the following conditions:
 \\(a) $\HH_{\beta}^{-1}(E)=0$ and  $H_{\beta}^{0}(E)$ is $\nu_{\alpha, \beta}$-semistable.
 \\(b)  $\HH_{\beta}^{-1}(E)$ is $\nu_{\alpha, \beta}$-semistable and  $H_{\beta}^{0}(E)$ is either $0$ or supported in dimension $0$. 
 
 Conversely, if $\HH_{\beta}^{-1}(E)$ is $\nu_{\alpha, \beta}$-stable, $\HH_{\beta}^{0}(E)$ is either $0$ or zero dimensional torsion sheaf, and $\mathrm{\mathrm{Hom}}(\mathcal{O}_{p}, E)=0$ for all points $p \in \mathbb{P}^{3}$, then $E $ is  $\lambda_{\alpha, \beta, s}$-stable for all $s\gg1$ sufficiently large. 
\end{lem}

\begin{rmk} \label{explanation}
Note that (a) and (b) apply to  semistable objects of Chern character $v$ with respect to the stability conditions on the left and right of the hyperbola $\H$, respectively. 
\end{rmk}

We also need the following Lemma in \cite{Huy}:

 \begin{lem} [{\cite[Corollary 11.4]{Huy}}]  \label{lemhuy} \label{exttrihuy}
  Let $j\colon Y\hookrightarrow X$ be a smooth hypersurface. Then for any $\mathcal{F} \in \mathrm{D}^{b}(Y)$ there exists a distinguished triangle
  $$\mathcal{F} \otimes \mathcal{O}_{Y}(-Y)[1] \rightarrow j^{*}j_{*}\mathcal{F} \rightarrow \mathcal{F}.$$
  \end{lem}

We have the following lemmas for more general cases:

\DeclarePairedDelimiter\ceil{\lceil}{\rceil}
\DeclarePairedDelimiter\floor{\lfloor}{\rfloor}

\begin{lem}[{\cite[Lemma 3.11]{R1}}] \label{Gfactor}
 Let $w=(0,1,-i-1/2, e)$. The stable objects of class $w$ for stability conditions with $\mathrm{Im}(Z_{\alpha, \beta, s}(w))<0$ near the hyperbola $\Imm(Z_{\alpha, \beta, s}(w))=0$ are of the form $\iota_{P_{*}}(\mathcal{I}_Z^{\vee})(- i)$, where $Z$ is a zero dimensional subscheme of length $l=1/6+(i/2)(i+1)-e$, and $P$ is a plane. 
\end{lem}

\begin{lem}[{\cite[Lemma 3.12]{R1}}] \label{Ffactor}
Let $w=(1,-1,-D+1/2, e)$ for $D=1,2$. The stable objects of Chern character $w$ for stability conditions with $\mathrm{Im}(Z_{\alpha, \beta, s}(w)) <0$ near the hyperbola  $\Imm(Z_{\alpha, \beta, s}(w))=0$ are sections of pure one-dimensional sheaves supported on degree $D$ curves with cokernel of length $l=3D-e-7/6 \in \mathbb{Z}_{\geq 0}$.
\end{lem}

\begin{lem} \label{torsionfactor}
Let $w=(1,-1,1/2, e)$. If $E$ is a stable object of Chern character $w$ for stability conditions with $\mathrm{Im}(Z_{\alpha, \beta, s}(w)) <0$ near the hyperbola  $\Imm(Z_{\alpha, \beta, s}(w))=0$, we have $e=-1/6$, and $E \cong \mathcal{O}_{\mathbb{P}^3}(-1)$.
\end{lem}

\begin{proof}

 Let $E$ be such an object with Chern character $w$, i.e., $E$ is stable near the hyperbola, which means it is still stable on the hyperbola  as we reach it from the right or from the left. For objects on the hyperbola $\H$, semistability does not change as $s$ varies; in particular, we can let $s \to +\infty$ and apply Lemma \ref{BMS8.9} to $E[1]$. Therefore, $\HH^1_\beta(E)$ is a torsion sheaf $T$ with zero-dimensional support, and  $\HH^0_\beta(E)$ is $\nu_{\alpha, \beta}$-semistable.  Notice that $E$ is semistable for $\Imm Z_{\alpha, \beta, s}(w) < 0$; this implies $\Hom(\OO_z[-1], E) = 0$ and so $\Hom(\OO_z[-1], \HH^0_\beta(E)) = 0$ for all $z \in \P^3$, as $\OO_z[-1]$ is semistable of phase $0$. But $(1,-1,1/2,e)$  is the Chern character of a tilt-semistable sheaf of the form $\mathcal{I}_{Z}(-1)$ where $Z$ is a zero dimensional subscheme of length $-1/6-e$ by Lemma  \ref{lem:lem_7}. But if $-1/6-e \neq 0$, then we would have $\Hom(\OO_z[-1], \HH^0_\beta(E)) \neq 0$ for $z \in Z$, which is a contradiction. Therefore $e=-1/6$ and thus $\HH^0_\beta(E) \cong \mathcal{O}_{\mathbb{P}^3}(-1)$. Now by Serre duality, we have
$$\Ext^1(\HH^1_\beta(E)[-1], \HH^0_\beta(E))= \Ext^2(T, \mathcal{O}_{\mathbb{P}^3}(-1))=\mathrm{H}^1(T)^{\vee}=0.$$
This means that  we must have $\HH^1_\beta(E)=0$; otherwise, $E$ would be a direct sum and hence unstable. Therefore  $E \cong \HH^0_\beta(E) \cong \mathcal{O}_{\mathbb{P}^3}(-1)$.

\end{proof}

Let $\H$ be the left branch of the hyperbola $\Imm(Z_{\alpha, \beta, s}(v))=0$. Corresponding to the walls in $\Stab^{\tilt}(\P^3)$, we list all (tilt and Bridgeland) strictly semistable objects on each wall defined by two given objects having the same slope which can be obtained from the extensions of the  pairs of objects:
\begin{thm} \label{main}
  For the walls in $\Stab^{\tilt}(\P^3)$, the walls on the both sides of $\H$ (and sufficiently close to it) are given by the following pairs:

\begin{center} 
\begin{tabular}{|m{3.9cm} | m{4.2cm} | m{5.2cm}|  } 
\hline
 walls in tilt stability & walls on the left side of $\H$ & 
walls on the right side of $\H$ \\ 
\hline
 \textcolor{magenta}{$ (\beta + 6.5)^{2}+\alpha^{2}=30.25$}
 
&
 
 \textcolor{red!85!blue}{ $\langle\mathcal{O}_{\mathbb{P}^3}(-1), \mathcal{I}_{Z'_{6}/P}(-6)\rangle$}
 
 \textcolor{magenta}{ $\langle\mathcal{I}_{Z_{1}}(-1), \mathcal{I}_{Z'_{5}/P}(-6)\rangle$} 
 
 \textcolor{magenta}{ $\langle\mathcal{I}_{Z_{2}}(-1), \mathcal{I}_{Z'_{4}/P}(-6)\rangle$}
 
 \textcolor{magenta}{ $\langle\mathcal{I}_{Z_{3}}(-1), \mathcal{I}_{Z'_{3}/P}(-6)\rangle$}
 
  \textcolor{magenta}{ $\langle\mathcal{I}_{Z_{4}}(-1), \mathcal{I}_{Z'_{2}/P}(-6)\rangle$}
  
   \textcolor{magenta}{ $\langle\mathcal{I}_{Z_{5}}(-1), \mathcal{I}_{Z'_{1}/P}(-6)\rangle$}
   
  \textcolor{magenta}{ $\langle\mathcal{I}_{Z_{6}}(-1), \mathcal{O}_{P}(-6)\rangle$}
 
 & \textcolor{red!85!blue}{$\langle\mathcal{O}_{\mathbb{P}^3}(-1), \iota_{P_{*}}\mathcal{I}_{Z_{6}}^{\vee}(-6)\rangle$} \\ 
\hline
 
  \textcolor{violet}{$(\beta + 5.5)^{2}+\alpha^{2}=18.25$} &

     \textcolor{violet!50!black}{ $\langle\mathcal{I}_{L_{0}}(-1), \mathcal{I}_{Z_{2}/P}(-5)\rangle$}  
     
   \textcolor{violet!55!blue}{ $\langle\mathcal{I}_{L_{1}}(-1), \mathcal{I}_{Z_{1}/P}(-5)\rangle$}
 
    \textcolor{violet!70!white}{ $\langle\mathcal{I}_{L_{2}}(-1), \mathcal{O}_{P}(-5)\rangle$} 
   
  &
   
   \textcolor{violet!70!white}{$\langle(\mathcal{O}_{\mathbb{P}^3}(-1) \rightarrow \mathcal{O}_{L}(1)), \mathcal{O}_{P}(-5)\rangle$}

    \textcolor{violet!55!blue}{$\langle(\mathcal{O}_{\mathbb{P}^3}(-1) \rightarrow \mathcal{O}_{L}),\iota_{P_{*}}\mathcal{I}_{Z_{1}}^{\vee}(-5)\rangle$}
    
     \textcolor{violet!50!black}{$\langle \mathcal{I}_{L_{0}}(-1),\iota_{P_{*}}\mathcal{I}_{Z_{2}}^{\vee}(-5)\rangle$}
    
    \\ 
\hline
\textcolor{green!64!black}{$(\beta + 4.5)^{2}+\alpha^{2}=8.25$} &
\textcolor{green!64!black}{ $\langle\mathcal{I}_{C_{2}}(-1), \mathcal{O}_{P}(-4)\rangle$} &
\textcolor{green!64!black}{ $\langle\mathcal{I}_{C_{2}}(-1), \mathcal{O}_{P}(-4)\rangle$}  \\ 
\hline
\textcolor{blue}{$(\beta +4)^{2}+\alpha^{2}=4$} & \textcolor{blue}{ $\langle\mathcal{O}_{\mathbb{P}^3}(-2), \mathcal{O}_{Q}(-3)\rangle$} & 
\textcolor{blue}{ $\langle\mathcal{O}_{\mathbb{P}^3}(-2), \mathcal{O}_{Q}(-3)\rangle$} \\ 
\hline
\end{tabular}
\end{center}
where $Z_{i}$'s and $Z'_{i}$'s are zero dimensional subschemes of length $i$,  $L_{i}$'s are lines plus $i$ extra embedded/floating points, $C_{2}$ is a conic, $Q$ is a quadric in $\mathbb{P}^{3}$,  $C$ a curve supported on $Q$,  and $\iota_{P}\colon P \hookrightarrow \mathbb{P}^{3}$ is the inclusion map. 
Every wall in $\Stab(\mathbb{P}^{3})$ that cuts the hyperbola ($\beta^2-\alpha^2=12$) is induced by a wall in $\Stab^{\tilt}(\mathbb{P}^{3})$.
 \end{thm}
\begin{proof}

The claims for the walls $(\beta +4)^{2}+\alpha^{2}=4$ and $(\beta + 4.5)^{2}+\alpha^{2}=8.25$, and the hyperbola are proven in \cite{R1}.

For the wall $ (\beta + 5.5)^{2}+\alpha^{2}=18.25 $, {\cite[Proposition 3.3]{R1}} implies the semistable objects on this semicircle are given by a pair $(F,G)$ such that $\mathrm{ch}^{-4}(F)=(1,3,7/2,e)$ and $\mathrm{ch}^{-4}(G)=(0,1,-3/2,5/3-e)$.  Therefore we have $\mathrm{ch}^{-4}(F \otimes \mathcal{O}_{\mathbb{P}^3}(-3))=(1,0,-1,e-3/2)$, which means $F \otimes \mathcal{O}_{\mathbb{P}^3}(-3) \cong \mathcal{I}_{L}(-4)$, where $L$ is a line with $l=5/2-e$ points (Lemma \ref{lem:lem_7}), and so $F \cong \mathcal{I}_{L}(-1)$. A similar argument and using Lemma \ref{lem:lem_7} shows $G \cong \mathcal{I}_{Z'/P}(-5)$, where $Z'$ is a zero dimensional subscheme of length $l'=-1/2+e$, supported on a plane $P \subset \mathbb{P}^{3}$. As the lengths $l, l'$ are non-negative, we obtain $1/2 \leq e \leq 5/2$. As $l,l'$ are integers, the following possibilities for $e,l$ and $l^{'}$ remain: (1) $e=5/2$, and so $l=0, l^{'}=2$. (2) $e=3/2$, and so $l=1, l^{'}=1$. (3) $e=1/2$, and so $l=2, l^{'}=0$. Thus all the possibilities for the pair $(F,G)$ are
$$
(F,G) \in \{
(\mathcal{I}_{L_{0}}(-1), \mathcal{I}_{Z'_{2}/P}(-5)),
(\mathcal{I}_{L_{1}}(-1), \mathcal{I}_{Z'_{1}/P}(-5)),
(\mathcal{I}_{L_{2}}(-1), \mathcal{I}_{Z'_{0}/P}(-5))\},
$$
where $Z'_{i}$'s are zero dimensional subschemes of length $i$, supported on plane $P$  in $\mathbb{P}^{3}$,  $L_{i}$'s are lines plus $i$ extra points in $\mathbb{P}^{3}$. Notice that we have $\mathcal{I}_{Z'_{0}/P}(-5)=\mathcal{O}_{P}(-5)$.

For the wall $ (\beta + 6.5)^{2}+\alpha^{2}=30.25 $, {\cite[Proposition 3.3]{R1}} implies that the semistable objects on this semicircle are given by a pair $(F,G)$ such that $\mathrm{ch}^{-4}(F)=(1,3,9/2,e)$ and $\mathrm{ch}^{-4}(G)=(0,1,-5/2,5/3-e)$.  Therefore, we have $\mathrm{ch}^{-4}(F \otimes \mathcal{O}_{\mathbb{P}^3}(-3))=(1,0,0,e-9/2)$, which means $F \otimes \mathcal{O}_{\mathbb{P}^3}(-3) \cong \mathcal{I}_{Z}(-4)$, where $Z$ is a zero dimensional subscheme of length $l=9/2-e$ (Lemma \ref{lem:lem_7}), and so $F \cong \mathcal{I}_{Z}(-1)$. A similar argument and using Lemma \ref{lem:lem_7} shows $G \cong \mathcal{I}_{Z'/P}(-6)$, where $Z'$ is a zero dimensional subscheme of length $l'=e+3/2$, supported on a plane $P \subset \mathbb{P}^{3}$. As the lengths $l, l'$ are non-negative integers, we must have  $-3/2 \leq e \leq 9/2$ with $e+1/2 \in \mathbb{Z}$. This leaves the following possibilities for the pair $(F,G)$:
\begin{equation*}
\begin{aligned}
(\mathcal{I}_{Z_{0}}(-1),\mathcal{I}_{Z'_{6}/P}(-6)),
(\mathcal{I}_{Z_{1}}(-1), \mathcal{I}_{Z'_{5}/P}(-6)),
(\mathcal{I}_{Z_{2}}(-1), \mathcal{I}_{Z'_{4}/P}(-6)),
(\mathcal{I}_{Z_{3}}(-1), \mathcal{I}_{Z'_{3}/P}(-6)),\\
(\mathcal{I}_{Z_{4}}(-1), \mathcal{I}_{Z'_{2}/P}(-6)),
(\mathcal{I}_{Z_{5}}(-1), \mathcal{I}_{Z'_{1}/P}(-6)),
(\mathcal{I}_{Z_{6}}(-1), \mathcal{I}_{Z'_{0}/P}(-6)),
\end{aligned}
\end{equation*}
where $Z_{i}$'s and $Z'_{i}$'s are zero dimensional subschemes of length $i$, and $Z'_{i}$'s are supported on  plane $P$  in $\mathbb{P}^{3}$. But we have $\mathcal{I}_{Z_{0}}(-1)=\mathcal{O}_{\mathbb{P}^3}(-1)$, and $\mathcal{I}_{Z'_{0}/P}(-6)=\mathcal{O}_{P}(-6)$. Hence, the proof for the left side of the hyperbola is completed.

As for the right side of the hyperbola, for the walls $(\beta +4)^{2}+\alpha^{2}=4$, $(\beta + 4.5)^{2}+\alpha^{2}=8.25$, and $ (\beta + 5.5)^{2}+\alpha^{2}=18.25 $, the claim is proven in \cite{R1}. As for the wall $ (\beta + 6.5)^{2}+\alpha^{2}=30.25 $, from Proposition  {\cite[Proposition 3.3]{R1}}, we have $\mathrm{ch}^{-4}_{\leq 2}(F)= (1,3,9/2)$ for a subobject or quotient of semistable objects with respect  stability conditions on the semicircle. Therefore we have  $\mathrm{ch}(F)= (1,-1,1/2,e)$. Lemma \ref{torsionfactor} implies  $e=-1/6$, and  $F \cong \mathcal{O}_{\mathbb{P}^3}(-1)$.  On the other hand, 
for the other corresponding JH factor we have $\mathrm{ch}(G)= (0,1,-6-1/2,15-e)$. Now using Lemma \ref{Gfactor}, we have $G \cong \iota_{P_{*}}\mathcal{I}_{Z_{l'}}^{\vee}(-6)$, where $Z_{l'}$ is a zero dimensional subscheme of length $l'=1/6+6+e=6$.

\begin{figure}[H]
\centering
\includegraphics[width=7.3cm]{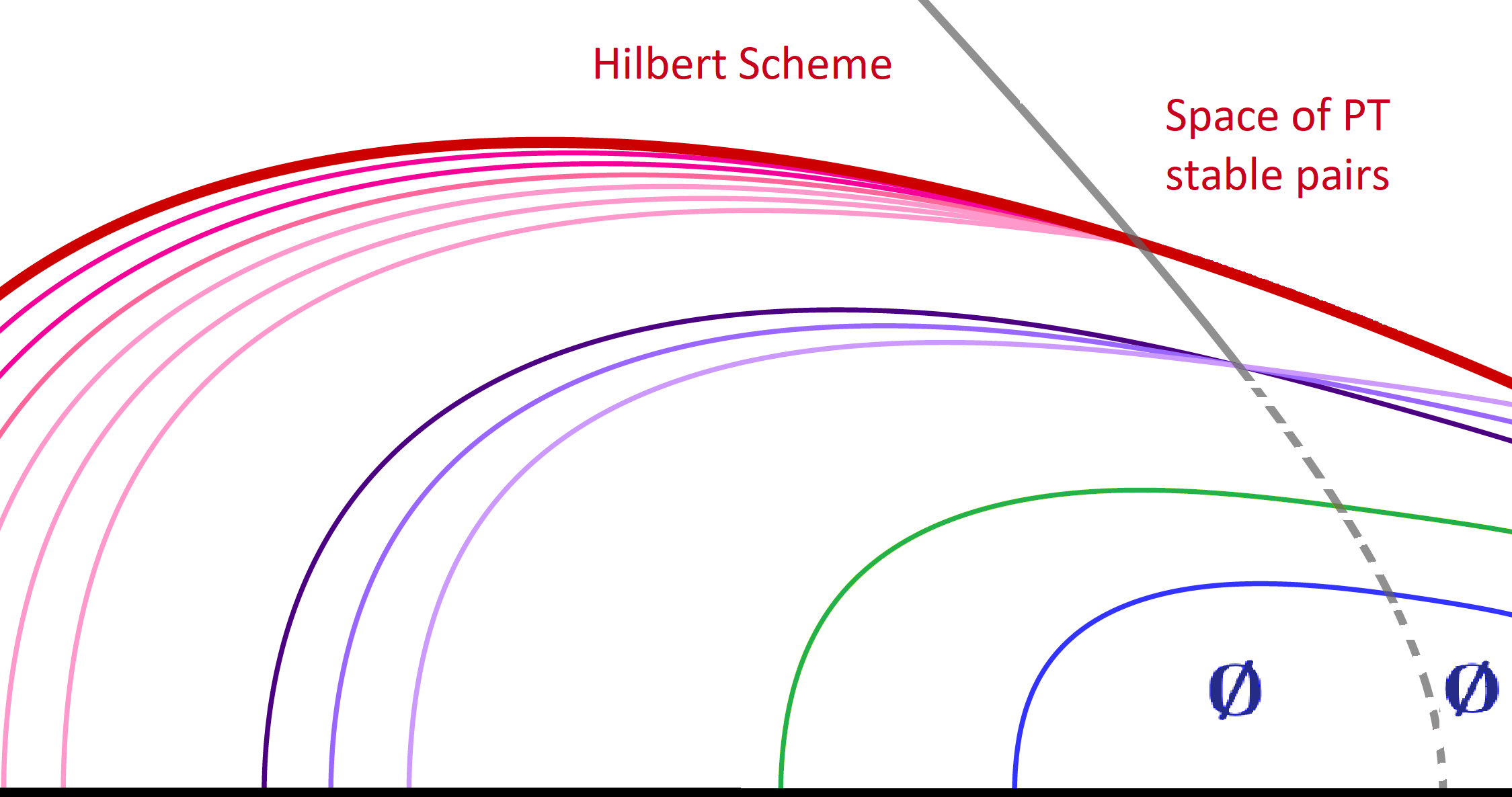}
\caption{Walls in Bridgeland stability space (in the $\alpha,\beta$ plane with $s=0.8$). Note that we are only interested in the region close enough to the hyperbola so that there is no other wall appears in the region}
\label{walls}
\end{figure}
\end{proof}

\section{Ext computation of the walls on the stable pairs side} \label{Ext computation of the walls on the stable pairs side}

So far, we have described all the possible walls in $\Stab(\mathbb{P}^3)$ close to the left branch of the hyperbola $\beta^2-\alpha^2=12$. To study the wall-crossing and describe the chambers (on the right side of the hyperbola) in Section \ref{chambers}, we need to compute all the necessary $\mathrm{Ext}^{1}$-groups associated with the walls. When there is no confusion, we use the notation  $A$ and $B$ for the  subobject and the quotient of the defining short exact sequence of any wall, respectively.

First, we have the following lemma to compute  the pull-backs (which will be also used in the next section):
\begin{lem} \label{pullback}

 For a plane $P$, a line $L$, $p=L \cap P$,  a line $L' \subset P$, zero-dimensional subschemes $q,q' \subset \P^3$, $L_1=L \cup q$, $L_2=L \cup q \cup q'$, a conic $C_{2}= L \cup L'$,  a zero-dimensional subscheme $Z$, and $\iota_{P}\colon P \hookrightarrow \mathbb{P}^{3}$  we have

\begin{equation*}
   \iota_{P}^{*}(\mathcal{I}_{L}) =
    \begin{cases}
     \mathcal{I}_{p/P}, &L \not\subset P.\\
      \mathcal{O}_{P}(-1) \oplus \mathcal{O}_{L}(-1), & L \subset P.
      
      \\
  
    \end{cases}
  \end{equation*}

  \begin{equation*}
   \iota_{P}^{*}(\mathcal{I}_{L_{1}}) = \iota_{P}^{*}(\mathcal{I}_{L \cup q}) =
    \begin{cases}
         \mathcal{I}_{p/P}, &L \not\subset P, q \not \subset P. \\
     \mathcal{I}_{p/P}, &L \not\subset P,  p=q. \\
     
        \mathcal{I}_{p \cup q} \oplus \OO_q, &L \not\subset P, p\neq q \subset P. \\
       
       \mathcal{O}_{P}(-1) \oplus \mathcal{O}_{L}(-1), & L \subset P, q \not \subset P.
      
      \\
     
       \mathcal{O}_{P}(-1) \oplus \mathcal{O}_{L}(-2) \oplus \OO_q, & L \subset P, q \subset L.
      
      \\
      \mathcal{I}_{q} (-1)\oplus \mathcal{O}_{q}, & L \subset P, q \subset P, q \not \subset L.
      
      \\
  
    \end{cases}
  \end{equation*}
  
   \begin{equation*}
   \iota_{P}^{*}(\mathcal{I}_{L_{2}}) = \iota_{P}^{*}(\mathcal{I}_{L \cup q \cup q'}) =
    \begin{cases}
         \mathcal{I}_{p/P}, &L \not\subset P, q,q' \not \subset P. \\
          \mathcal{I}_{p \cup q} \oplus \OO_q, &L \not\subset P, q \subset P,q' \not \subset P. \\
           \mathcal{I}_{p \cup q \cup q'} \oplus \OO_q \oplus \OO_{q'}, &L \not\subset P, q,q' \subset P; q,q'\neq p. \\
     \mathcal{I}_{p \cup q'} \oplus \OO_{q'} , &L \not\subset P;  q,q' \subset P; \\& q'\neq p=q. \\
     
        \mathcal{I}_{p/P} ,  &L \not\subset P;  q,q' \subset P;  \\&q'= p=q.\\
       
       \mathcal{O}_{P}(-1) \oplus \mathcal{O}_{L}(-1), & L \subset P, q,q' \not \subset P.
      
      \\
     
       \mathcal{O}_{P}(-1) \oplus \mathcal{O}_{L}(-1) \oplus \OO_q, & L \subset P, q \subset L, q' \not \subset P.
        \\
     
       \mathcal{O}_{P}(-1) \oplus \mathcal{O}_{L}(-1) \oplus \OO_q \oplus \OO_{q'}, & L \subset P, q,q' \subset L.
      \\

    \end{cases}
  \end{equation*}

\begin{equation*}
   \iota_{P}^{*}(\mathcal{O}_{\mathbb{P}^3} \xrightarrow{s} \mathcal{O}_{L}(1)) =
    \begin{cases}
     
     \mathcal{I}_{L\cap P}, &L\not\subset P,\\
      &  \text{zero  locus   of  $s$  is  not   $L \cap P$}.\\
   \mathcal{O}_{P} \oplus \mathcal{O}_{L \cap P}[-1], &  L\not\subset P,\\
   &  \text{zero  locus   of  $s$  is $ L \cap P$}.\\
    
      (\mathcal{O}_{P} \rightarrow \mathcal{O}_{L}(1)) \oplus \mathcal{O}_{L}, & L \subset P.
      
      \\
  
    \end{cases}
  \end{equation*}

\begin{equation*}
   \iota_{P}^{*}(\mathcal{O}_{\mathbb{P}^3} \xrightarrow{s} \mathcal{O}_{L}(2)) =
    \begin{cases}
     
       \mathcal{I}_{L\cap P}, &L \not\subset P,\\
      &  \text{zero  locus   of  $s$ does not contain $L \cap P$}.\\
   \mathcal{O}_{P} \oplus \mathcal{O}_{L \cap P}[-1], &  L \not\subset P,\\
   &  \text{zero  locus   of  $s$  contains $L \cap P$}.\\
    
      (\mathcal{O}_{P} \rightarrow \mathcal{O}_{L}(2)) \oplus \mathcal{O}_{L}(1), & L \subset P.
      
      \\
  
    \end{cases}
  \end{equation*}

\end{lem}

\begin{proof} The equations for $ \iota_{P}^{*}(\mathcal{I}_{L})$ are proven in \cite{R1}. The proofs for $ \iota_{P}^{*}(\mathcal{I}_{L_1})$ and $ \iota_{P}^{*}(\mathcal{I}_{L_2})$ are similar.

Now, let us compute $\iota_{P}^{*}(\mathcal{O}_{\mathbb{P}^3} \rightarrow \mathcal{O}_{L}(1))$. First, assume that $L \not \subset P$, and the zero locus of $s$ is not $L \cap P$. Then we  have  $\iota_{P}^{*}(\mathcal{O}_{\mathbb{P}^3} \rightarrow \mathcal{O}_{L}(1))= (\mathcal{O}_{P} \rightarrow \mathcal{O}_{L \cap P})= \mathcal{I}_{L \cap P}$. Secondly, assume that $L \not \subset P$, and the zero locus of $s$ is  $L \cap P$. Then from the sequence $\mathcal{O}_{L}(1)[-1]   \rightarrow (\mathcal{O}_{\mathbb{P}^3} \rightarrow \mathcal{O}_{L}(1)) \rightarrow \mathcal{O}_{\mathbb{P}^3}$, we have $\iota_{P}^{*}(\mathcal{O}_{\mathbb{P}^3} \rightarrow \mathcal{O}_{L}(1))= \mathcal{O}_{P} \oplus \mathcal{O}_{L \cap P}[-1]$.

Now, let us assume  $L  \subset P$. In this case, we have $\iota_{P}^{*}(\mathcal{O}_{\mathbb{P}^3} \rightarrow \mathcal{O}_{L}(1))= \iota_{P}^{*}(\mathcal{O}_{\mathbb{P}^3} \rightarrow \mathcal{I}_{L}(1) \rightarrow \mathcal{O}_{\mathbb{P}^3}(1))= \mathcal{O}_{P} \rightarrow \iota_{P}^{*}(\mathcal{I}_{L}(1)) \rightarrow \mathcal{O}_{P}(1)$. Hence from the first part we obtain $\mathcal{O}_{P} \rightarrow \mathcal{O}_{L} \oplus \mathcal{O}_{P} \rightarrow \mathcal{O}_{P}(1)= (\mathcal{O}_{P} \rightarrow \mathcal{O}_{L}) \oplus \mathcal{O}_{L}(1)=  (\mathcal{O}_{P} \rightarrow \mathcal{O}_{L}(1)) \oplus \mathcal{O}_{L}$.

A similar argument implies the result for $\iota_{P}^{*}(\mathcal{O}_{\mathbb{P}^3} \rightarrow \mathcal{O}_{L}(2))$.

\end{proof}

We have the following Lemmas on the  $\mathrm{Ext}^{1}$'s:

\begin{lem} \label{ext1green}
For the wall  $\langle\mathcal{I}_{C_{2}}(-1), \mathcal{O}_{P}(-4)\rangle$, we have:

\begin{eqnarray*}
\mathrm{Ext}^{1}(\mathcal{I}_{C_{2}}(-1),\mathcal{I}_{C_{2}}(-1))=\C^8,&&
\mathrm{Ext}^{1}(\mathcal{O}_{P}(-4), \mathcal{O}_{P}(-4))=\C^3,\\
\mathrm{Ext}^{1}(\mathcal{O}_{P}(-4),\mathcal{I}_{C_{2}}(-1))=\C^{13},&&
\mathrm{Ext}^{1}(\mathcal{I}_{C_{2}}(-1), \mathcal{O}_{P}(-4))=\C.
\end{eqnarray*}

\end{lem}

\begin{proof}
This is just a straightforward computation.
\end{proof}

\begin{lem}[{\cite[Lemma 4.3]{R1}}] ~\label{ext1purple1}
For the wall $\langle \mathcal{I}_{L}(-1),\iota_{P_{*}}(\mathcal{I}_{Z_{2}}^{\vee}(-5)\rangle$, we have:

$$
\mathrm{Ext}^{1}(\mathcal{I}_{L}(-1),\mathcal{I}_{L}(-1))= \mathbb{C}^{4},\quad
\mathrm{Ext}^{1}(\iota_{P_{*}}(\mathcal{I}_{Z_{2}}^{\vee}(-5), \iota_{P_{*}}(\mathcal{I}_{Z_{2}}^{\vee}(-5))= \mathbb{C}^{7},$$$$
\mathrm{Ext}^{1}(\iota_{P_{*}}(\mathcal{I}_{Z_{2}}^{\vee}(-5),\mathcal{I}_{L}(-1))= \mathbb{C}^{18},
$$
\begin{equation*}
    \mathrm{Ext}^{1}(\mathcal{I}_{L}(-1),\iota_{P_{*}}(\mathcal{I}_{Z_{2}}^{\vee}(-5))=
    \begin{cases}
    0, & \langle Z_2 \rangle \cap L = \emptyset\\
       \mathbb{C}, & \text{ $\langle Z_2 \rangle \cap L \neq \emptyset$  but $\langle Z_2 \rangle \neq L$}\\
     
      \mathbb{C}^{2}, &  \langle Z_2 \rangle =  L

    \end{cases}
  \end{equation*}
where $\langle Z_2 \rangle$ is the line spanned by $Z_{2}$.

\end{lem}

In the statement of the following Lemma, all tensor products are in  $\mathrm{D}^{b}(P)$:

\begin{lem} \label{tensor}
Let $p \neq q$ be two different points in a plane $P$, such that $p \subset l$, and  $p \not \subset l'$, where lines $l,l'$ are in $P$. Then $\II_{p} \otimes \II_p$ fits into the short exact sequence  $\OO_p \hookrightarrow \II_{p} \otimes \II_p \twoheadrightarrow{}  \II^2_{p}$. Also, $\II_{p \cup q} \otimes \II_p$ fits into the short exact sequence  $\OO_p \hookrightarrow \II_{p \cup q} \otimes \II_p \twoheadrightarrow{}  \II_{p^2 \cup q}$. Furthermore, we have
\begin{eqnarray*}
 \II_p \otimes \OO_p=\OO_p[1] \oplus \OO_p^{\oplus 2},\quad \II_p \otimes \OO_q=\OO_q, \quad
\II_p \otimes \II_q=\II_{p \cup q},\\ 
\quad \II_p \otimes \OO_l=\OO_l(-1) \oplus \OO_p, \quad
\II_p \otimes \OO_{l'}=\OO_{l'}.
\end{eqnarray*}

\end{lem}
\begin{proof}
All the cases either proved in {\cite[Lemma 4.2]{R1}} or exactly a similar argument there implies the result. 

\end{proof}

\begin{lem} ~\label{ext1purple2}
For the wall $\langle\mathcal{O}_{\mathbb{P}^3}(-1) \rightarrow \mathcal{O}_{L},\iota_{P_{*}}\mathcal{I}_{Z_{1}}^{\vee}(-5)\rangle$, we have 

$$
\mathrm{Ext}^{1}((\mathcal{O}_{\mathbb{P}^3}(-1) \rightarrow \mathcal{O}_{L}), (\mathcal{O}_{\mathbb{P}^3}(-1) \rightarrow \mathcal{O}_{L}))=\mathbb{C}^{5}, \quad
\mathrm{Ext}^{1}(\iota_{P_{*}}\mathcal{I}_{Z_{1}}^{\vee}(-5), \iota_{P_{*}}\mathcal{I}_{Z_{1}}^{\vee}(-5))=\mathbb{C}^{5},$$

\begin{equation*}
  \mathrm{Ext}^{1}(\iota_{P_{*}}\mathcal{I}_{Z_{1}}^{\vee}(-5), \mathcal{O}_{\mathbb{P}^3}(-1) \xrightarrow{s} \mathcal{O}_{L}) =
    \begin{cases}
     \mathbb{C}^{21}, &\text{the zero  locus   of  $s$  is at $Z_1$,}\\
      
     \mathbb{C}^{20}, &    \text{the  zero  locus   of  $s$  is in  $ P$ but $\neq Z_1$}, \\
      
   \mathbb{C}^{19}, &  \text{the zero  locus   of  $s$  is not in $ P$}
    \end{cases}
  \end{equation*}
\begin{equation*}
    \mathrm{Ext}^{1}(\mathcal{O}_{\mathbb{P}^3}(-1) \xrightarrow{s} \mathcal{O}_{L},\iota_{P_{*}}\mathcal{I}_{Z_{1}}^{\vee}(-5))=
    \begin{cases}
      \mathbb{C}^{2}, & 
         \text{the zero  locus   of  $s$  is at $Z_1$},\\
     
      \mathbb{C}, &    \text{ the  zero  locus   of  $s$  is in  $ P$ but $\neq Z_1$}, 
         \\
   
      0, & \text{the zero  locus   of  $s$  is not in $ P$}
    \end{cases}
  \end{equation*}
\end{lem}

\begin{proof}
We notice that $\mathrm{Ext}^{1}(\mathcal{O}_{\mathbb{P}^3}(-1) \rightarrow \mathcal{O}_{L}, \mathcal{O}_{\mathbb{P}^3}(-1) \rightarrow \mathcal{O}_{L})$  generically parametrizes a line in $\mathbb{P}^{3}$ (given by $\mathbb{G}r(2,4)$) 
together with a choice of one point on the line. Thus we have $\mathrm{Ext}^{1}(\mathcal{O}_{\mathbb{P}^3}(-1) \rightarrow \mathcal{O}_{L}, \mathcal{O}_{\mathbb{P}^3}(-1) \rightarrow \mathcal{O}_{L})= \mathbb{C}^{5}$. Also, $\mathrm{Ext}^{1}(\iota_{P_{*}}\mathcal{I}_{Z_{1}}^{\vee}(-5), \iota_{P_{*}}\mathcal{I}_{Z_{1}}^{\vee}(-5))$ is the parameter space of one point in a plane, which gives the flag variety $\mathfrak{Fl}_{1}$. Therefore we have  $\mathrm{Ext}^{1}(B,B)=\mathbb{C}^{5}$.

Let $q$ be the zero locus of $s$. Now we compute $\mathrm{Ext}^{1}(B,A)$ as follows:
$$\mathrm{Ext}^{1}(\iota_{P_{*}}\mathcal{I}_{Z_{1}}^{\vee}(-5),\mathcal{O}_{\mathbb{P}^3}(-1) \rightarrow \mathcal{O}_{L})=\mathrm{H}^{0}\bigl(\iota_{P}^{*}(\mathcal{O}_{\mathbb{P}^3} \rightarrow \mathcal{O}_{L}(1)) \otimes \mathcal{I}_{Z_{1}}(5)\bigr).$$

\begin{description}

\item [\underline{First, assume that $L \not \subset P$}]

 There are two cases:
\begin{itemize}
    \item (I) Zero  locus   of  $s$  is not at  $L \cap P$: Using Lemmas \ref{pullback} and \ref{tensor} we have:

(i) if  $L \cap P= Z_{1}$, then we have $\mathrm{H}^{0}(\iota_{P}^{*}(\mathcal{O}_{\mathbb{P}^3} \rightarrow \mathcal{O}_{L}(1)) \otimes \mathcal{I}_{Z_{1}}(5))= \mathrm{H}^{0}(\mathcal{I}_{L \cap P} \otimes \mathcal{I}_{Z_{1}}(5))= \mathrm{H}^{0}(\mathcal{I}_{Z_{1}} \otimes \mathcal{I}_{Z_{1}}(5))=\mathrm{H}^{0}(\OO_{Z_1}) \oplus \mathrm{H}^{0}(\II_{Z_1}^{2}(5))= \mathbb{C}^{19}$.

(ii) if $L  \cap P \neq Z_{1}$, then we have $\mathrm{H}^{0}(\iota_{P}^{*}(\mathcal{O}_{\mathbb{P}^3} \rightarrow \mathcal{O}_{L}(1)) \otimes \mathcal{I}_{Z_{1}}(5)))=\mathrm{H}^{0}(\mathcal{I}_{L \cap P} \otimes \mathcal{I}_{Z_{1}}(5))= \mathrm{H}^{0}(\mathcal{I}_{(L \cap P) \cup Z_{1} }(5))= \mathbb{C}^{19}$.

\item (II) Zero  locus   of  $s$  is  $L \cap P$: By Lemma \ref{pullback}, we have $\mathrm{H}^{0}(\iota_{P}^{*}(\mathcal{O}_{\mathbb{P}^3} \rightarrow \mathcal{O}_{L}(1)) \otimes \mathcal{I}_{Z_{1}}(5))= \mathrm{H}^{0}((\mathcal{O}_{P} \oplus \mathcal{O}_{L \cap P}[-1]) \otimes \mathcal{I}_{Z_{1}}(5))= \mathrm{H}^0(\mathcal{O}_{P} \otimes \mathcal{I}_{Z_{1}}(5)) \oplus \mathrm{H}^{-1}(\mathcal{O}_{L\cap P} \otimes \mathcal{I}_{Z_{1}}(5))=\mathrm{H}^0(\mathcal{I}_{Z_{1}}(5)) \oplus \mathrm{H}^{-1}(\mathcal{O}_{L\cap P} \otimes \mathcal{I}_{Z_{1}}(5))=\C^{20} \oplus \mathrm{H}^{-1}(\mathcal{O}_{L\cap P} \otimes \mathcal{I}_{Z_{1}}(5))$. Now there are two subcases:
 
(i) if $Z_1\neq P \cap L$, by Lemma \ref{tensor} we have $\mathrm{Ext}^1(B,A)=\C^{20} \oplus \mathrm{H}^{-1}(\mathcal{O}_{L\cap P} \otimes \mathcal{I}_{Z_{1}}(5))=\C^{20} \oplus \mathrm{H}^{-1}(\mathcal{O}_{L\cap P})=\C^{20} \oplus 0 =\mathbb{C}^{20}$. 
 
 (ii) if $Z_1= P \cap L$, by Lemma \ref{tensor} we have $\mathcal{O}_{L \cap P} \otimes \mathcal{I}_{Z_{1}}(1)=\OO_{Z_1}[1] \oplus \OO^{\oplus 2}_{Z_1}$, and therefore $\mathrm{Ext}^1(B,A)=\mathrm{H}^0(\mathcal{I}_{Z_{1}}(5)) \oplus \mathrm{H}^0(\mathcal{O}_{Z_{1}}) = \mathbb{C}^{20} \oplus \C^{1}=\C^{21}.$
 \end{itemize}

\item [\underline{Second, assume that $L \subset P$}] In this case we have $\iota_{P}^{*}(\mathcal{O}_{\mathbb{P}^3} \rightarrow \mathcal{O}_{L}(1))=(\mathcal{O}_{P} \rightarrow \mathcal{O}_{L}(1))   \oplus \mathcal{O}_{L}$. Therefore we have $$\Ext^1(B,A)=\mathrm{H}^{0}(\iota_{P}^{*}(\mathcal{O}_{\mathbb{P}^3} \rightarrow \mathcal{O}_{L}(1)) \otimes \mathcal{I}_{Z_{1}}(5))= \mathrm{H}^{0}((\mathcal{O}_{P}(5) \rightarrow \mathcal{O}_{L}(6))\otimes \mathcal{I}_{Z_{1}})\oplus \mathrm{H}^{0}(\mathcal{O}_{L}\otimes \mathcal{I}_{Z_{1}}(5))$$$$=\mathrm{H}^{0}(\mathcal{I}_{Z_1}(5) \rightarrow \mathcal{O}_{L}\otimes \mathcal{I}_{Z_{1}}(6))\oplus \mathrm{H}^{0}(\mathcal{O}_{L}\otimes \mathcal{I}_{Z_{1}}(5)).$$
But using \ref{tensor} we have 

\begin{equation*}
     \mathrm{H}^{0}(\mathcal{O}_{L}\otimes \mathcal{I}_{Z_{1}}(5))=
    \begin{cases}
     \mathrm{H}^{0}(\mathcal{O}_{L}(5)) \cong \mathbb{C}^{6}, & Z_{1} \not\subset L,\\
     \mathrm{H}^{0}(\mathcal{O}_{L}(4) \oplus \mathcal{O}_{Z_{1}})=   \mathbb{C}^5 \oplus \mathbb{C} \cong \mathbb{C}^{6}, & Z_{1} \subset L.

    \end{cases}
  \end{equation*}
On the other hand, to compute $\mathrm{H}^{0}(\mathcal{I}_{Z_1}(5) \xrightarrow{s} \mathcal{O}_{L}\otimes \mathcal{I}_{Z_{1}}(6))=\mathrm{H}^{0}(\ker(s))$, we notice that $s$ factors as follows: $\II_{Z_1}(5)  \xrightarrow{s_1}  \OO_L \otimes \II_{Z_1}(5) \xrightarrow{s_2} \OO_{L} \otimes \II_{Z_1}(6).$
Note that $\ker(s_1)=\II_{Z_1}(4)$. Now there are three cases:
\begin{itemize}
 
\item (1) If $Z_1 \not \subset L$, then by Lemma \ref{tensor} we have   $s_2\colon \OO_L(5) \to \OO_L(6)$, which is injective.
 \item (2) If $Z_1 \subset L$, but $Z_1 \neq \text{zero locus of $s$}$, then by Lemma \ref{tensor} we have $s_2\colon \OO_L(4) \oplus \OO_{Z_1} \to \OO_L(5) \oplus \OO_{Z_1}$, which is injective.
 \item (3) If $Z_1= \text{zero locus of $s \subset L$}$, then $s_2 |_{\OO_{Z_1}}=0$.
\end{itemize}
For cases (1) and (2), where $s_2$ is injective, we have
$$\mathrm{H}^{0}(\ker(s))=\mathrm{H}^{0}(\ker(s_1))=\mathrm{H}^0(\II_{Z_1}(4))=\C^{14}.$$
For case (3), similarly using Lemma \ref{tensor}, we have
$$\mathrm{H}^{0}(\ker(s))=\mathrm{H}^{0}(\ker(\II_{Z_1}(5) \to \OO_L(4)))=\mathrm{H}^{0}(\OO_P(4))=\C^{15}.$$
Thus $\Ext^1(B,A)$ is 
\begin{equation*}
     \mathrm{H}^{0}(\mathcal{I}_{Z_1}(5) \rightarrow \mathcal{O}_{L}\otimes \mathcal{I}_{Z_{1}}(6))\oplus \mathrm{H}^{0}(\mathcal{O}_{L}\otimes \mathcal{I}_{Z_{1}}(5))=
    \begin{cases}
     \C^{14}\oplus \C^6=\C^{20}, & Z_1\neq q,\\
     \C^{15} \oplus \C^6=\C^{21}, & Z_1=q \subset L.

    \end{cases}
  \end{equation*}

\end{description}

For the last part, using Serre duality we have
$$\mathrm{Ext}^{1}(A,B)= \mathrm{Ext}^{1}(\mathcal{O}_{\mathbb{P}^3}(-1) \rightarrow \mathcal{O}_{L},\iota_{P_{*}}\mathcal{I}_{Z_{1}}^{\vee}(-5))= \mathrm{Ext}^{1}((\iota_{P}^{*}(\mathcal{O}_{\mathbb{P}^3}(-1) \rightarrow \mathcal{O}_{L}),\mathcal{I}_{Z'_{1}}^{\vee}(-5))$$$$=\mathrm{Ext}^{1}(\mathcal{I}_{Z'_{1}}^{\vee}(-5), (\iota_{P}^{*}(\mathcal{O}_{\mathbb{P}^3}(-1) \rightarrow \mathcal{O}_{L}) \otimes \OO_{\mathbb{P}^3}(-3))^{\vee}=\mathrm{H}^1((\iota_{P}^{*}(\mathcal{O}_{\mathbb{P}^3}(-1) \rightarrow \mathcal{O}_{L}) \otimes \II_{Z_1}^{\vee}(2)).$$

There are three cases:

\begin{description}

\item [\underline{(1) $L \not \subset P$ and  $L\cap P$ is the  zero locus of $s$}]: From Lemmas \ref{pullback} and \ref{tensor}, we have
$$\mathrm{Ext}^{1}(A,B)=\mathrm{H}^1((\iota_{P}^{*}(\mathcal{O}_{\mathbb{P}^3} \rightarrow \mathcal{O}_{L}(1)) \otimes \II_{Z_1}^{\vee}(1))=\mathrm{H}^{1}(\II_{Z_1}(1))^{\vee} \oplus \mathrm{H}^{0}(\OO_{L \cap P} \otimes \II_{Z_1}(1))^{\vee}$$
\begin{equation*}
   =0 \oplus \mathrm{H}^{0}(\OO_{L \cap P} \otimes \II_{Z_1}(1))^{\vee}=
    \begin{cases}
     \mathrm{H}^{0}(\OO_{P \cap L})^{\vee} =\C, & L \cap P \neq Z_{1}\\
    \mathrm{H}^{0}(\OO_{L \cap P}[1] \oplus \OO_{L \cap P}^{\oplus 2})^{\vee}=0 \oplus \C^2 =\C^2,  &  L \cap P=Z_{1}
    \\

    \end{cases}
  \end{equation*}

\item [\underline{ (2) $L \not \subset P$, and  $L \cap P$ is not  the  zero section of $s$}]: From Lemmas \ref{pullback} and \ref{tensor}, we have
 $$\mathrm{Ext}^{1}(A,B)= 
\mathrm{H}^{1}(\mathcal{I}_{L \cap P} \otimes \mathcal{I}_{Z_{1}}(1))^{\vee}$$
\begin{equation*}
   =
    \begin{cases}
     \mathrm{H}^{1}(\II_{(P \cap L) \cup Z_1 }(1))^{\vee}=0, & L \cap P \neq Z_{1}\\
    \mathrm{H}^{1}(\OO_{L \cap P})^{\vee} \oplus \mathrm{H}^{1}(\II_{L \cap P}^2(1))^{\vee}=0 \oplus 0 =0,  &  L \cap P=Z_{1}
    \\

    \end{cases}
  \end{equation*}

\item [\underline{(3) $L \subset P$}]:
Again using Lemmas \ref{pullback}, we have
$$\mathrm{Ext}^{1}(A,B)=\mathrm{H}^{1}((\mathcal{O}_{P} \xrightarrow{s} \mathcal{O}_{L }(1))\otimes \mathcal{I}_{Z_{1}}(1))^{\vee} \oplus \mathrm{H}^{1}(\mathcal{O}_{L}\otimes \mathcal{I}_{Z_{1}}(1))^{\vee}$$$$=\mathrm{H}^{1}((\mathcal{O}_{P} \xrightarrow{s} \mathcal{O}_{L }(1))\otimes \mathcal{I}_{Z_{1}}(1))^{\vee} \oplus 0.$$
Exactly the same argument as above (just twisting everything by $-4$) implies
\begin{equation*}
     \mathrm{H}^{0}(\mathcal{I}_{Z_1}(1) \rightarrow \mathcal{O}_{L}\otimes \mathcal{I}_{Z_{1}}(2))=
    \begin{cases}
    0, & Z_1\neq \text{zero locus of $s$},\\
     \C, & Z_1= \text{zero locus of $s \subset L$}.

    \end{cases}
  \end{equation*}
Now, taking the cohomology long exact sequence of $\mathcal{O}_{L} \otimes \mathcal{I}_{Z_{1}}(2)[-1] \rightarrow (\mathcal{O}_{P} \rightarrow \mathcal{O}_{L }(1))\otimes \mathcal{I}_{Z_{1}}(1)  \rightarrow \mathcal{O}_{P} \otimes \mathcal{I}_{Z_{1}}(1)$ and noticing that $\mathrm{H}^{0}(\mathcal{O}_{P} \otimes \mathcal{I}_{Z_{1}}(1))=\mathrm{H}^{0}(\mathcal{I}_{Z_{1}}(1))=\C^2,$
and (by Lemma \ref{tensor})
\begin{equation*}
    \mathrm{H}^{0}(\mathcal{O}_{L} \otimes \mathcal{I}_{Z_{1}}(2)) =
    \begin{cases}
     \mathrm{H}^{0}(\mathcal{O}_{L}(2)) \cong \mathbb{C}^{3}, & Z_{1} \not\subset L,\\
     \mathrm{H}^{0}(\mathcal{O}_{L}(1) \oplus \mathcal{O}_{Z_{1}})=   \mathbb{C}^{2} \oplus \mathbb{C} \cong \mathbb{C}^{3}, & Z_{1} \subset L,

    \end{cases}
  \end{equation*}
we have
\begin{center}
\begin{tikzpicture}[descr/.style={fill=white,inner sep=1.5pt}]
        \matrix (m) [
            matrix of math nodes,
            row sep=1em,
            column sep=2.5em,
            text height=1.5ex, text depth=0.25ex
        ]
        {  & 0 & \mathrm{H}^{0}((\mathcal{O}_{P} \rightarrow \mathcal{O}_{L }(1))\otimes \mathcal{I}_{Z_{1}}(1)) & \C^2\\
            &  \mathbb{C}^{3} & \mathrm{H}^{1}((\mathcal{O}_{P} \rightarrow \mathcal{O}_{L }(1))\otimes \mathcal{I}_{Z_{1}}(1))& 0.\\
        };

        \path[overlay,->, font=\scriptsize,>=latex]

        (m-1-2) edge (m-1-3)
        (m-1-3) edge (m-1-4)
        (m-1-4) edge[out=355,in=175] node[descr,yshift=0.3ex] {} (m-2-2)
        (m-2-2) edge (m-2-3)
        (m-2-3) edge (m-2-4)
     ;
\end{tikzpicture}
\end{center}
 Therefore, in this case we have
 \begin{equation*}
     \Ext^1(A,B)=\mathrm{H}^{1}(\mathcal{I}_{Z_1}(1) \rightarrow \mathcal{O}_{L}\otimes \mathcal{I}_{Z_{1}}(2))^{\vee}=
    \begin{cases}
    \C, & Z_1\neq \text{zero locus of $s$},\\
     \C^2, & Z_1= \text{zero locus of $s \subset L$}.

    \end{cases}
  \end{equation*}

\end{description}

\end{proof}

\begin{lem} ~\label{ext1purple3}
For the wall $\langle\mathcal{O}_{\mathbb{P}^3}(-1) \rightarrow \mathcal{O}_{L}(1), \mathcal{O}_{P}(-5)\rangle$, we have:

$$\mathrm{Ext}^{1}(\mathcal{O}_{\mathbb{P}^3}(-1) \rightarrow \mathcal{O}_{L}(1),\mathcal{O}_{\mathbb{P}^3}(-1) \rightarrow \mathcal{O}_{L}(1))=\mathbb{C}^{6},\quad
\mathrm{Ext}^{1}(\mathcal{O}_{P}(-5),\mathcal{O}_{P}(-5))= \mathbb{C}^{3},$$

\begin{equation*}
  \mathrm{Ext}^{1}(\mathcal{O}_{P}(-5), \mathcal{O}_{\mathbb{P}^3}(-1) \xrightarrow{s} \mathcal{O}_{L}(1)) =
    \begin{cases}
     \mathbb{C}^{22}, &\text{ $ L \subset P$}\\
     \mathbb{C}^{21}, &\text{$L\not\subset P$,  and the zero  locus   of  $s \supset L \cap P$},\\
   \mathbb{C}^{20}, &  \text{$L\not\subset P$, and the zero  locus   of  $s \not \supset L \cap P$}

    \end{cases}
  \end{equation*}

\begin{equation*}
    \mathrm{Ext}^{1}(\mathcal{O}_{\mathbb{P}^3}(-1) \xrightarrow{s} \mathcal{O}_{L}(1),\mathcal{O}_{P}(-5))=
    \begin{cases}
      \mathbb{C}^{2}, &    L \subset P, 
      
      \\
      \mathbb{C}, & \text{$L \not\subset P$  and the zero  locus   of  $s \supset L \cap P$}
      \\
      0, & \text{$L\not\subset P$, and the zero  locus   of  $s \not \supset L \cap P$}
    \end{cases}
  \end{equation*}

\end{lem}

\begin{proof}
 We notice that $\mathrm{Ext}^{1}((\mathcal{O}_{\mathbb{P}^3} \rightarrow \mathcal{O}_{L}(2)) \otimes \mathcal{O}_{\mathbb{P}^3}(-1), (\mathcal{O}_{\mathbb{P}^3} \rightarrow \mathcal{O}_{L}(2)) \otimes \mathcal{O}_{\mathbb{P}^3}(-1))$  generically parametrizes a line in $\mathbb{P}^{3}$ (given by $\mathbb{G}r(2,4)$) 
together with a choice of two points on the line. Thus we have $\mathrm{Ext}^{1}(\mathcal{O}_{\mathbb{P}^3}(-1) \rightarrow \mathcal{O}_{L}(1), \mathcal{O}_{\mathbb{P}^3}(-1) \rightarrow \mathcal{O}_{L}(1))= \mathbb{C}^{6}$. Also, we note that $\mathrm{Ext}^{1}(\iota_{P_{*}}(\mathcal{O}_{P})(-5), \iota_{P_{*}}(\mathcal{O}_{P})(-5))$ is the parameter space of a plane in  $\mathbb{P}^{3}$ which is given by $(\mathbb{P}^{3})^{*}$, and thus   $\mathrm{Ext}^{1}(B,B)=\mathbb{C}^{3}$.

Let $q \cup q'$ be the zero locus of $s$. Now we compute $\mathrm{Ext}^{1}(B,A)$ as follows: $$\mathrm{Ext}^{1}(B,A)=\mathrm{Ext}^{1}(\iota_{P_{*}}(\mathcal{O}_{P})(-5),\mathcal{O}_{\mathbb{P}^3}(-1) \rightarrow \mathcal{O}_{L}(1))=\mathrm{Ext}^{1}(\mathcal{O}_{P}(-5),\iota_{P}^{!}(\mathcal{O}_{\mathbb{P}^3}(-1) \rightarrow \mathcal{O}_{L}(1)))$$$$
=\mathrm{Hom}(\mathcal{O}_{P}(-5),\iota_{P}^{*}(\mathcal{O}_{\mathbb{P}^3} \rightarrow \mathcal{O}_{L}(2)))=\mathrm{H}^{0}(\iota_{P}^{*}(\mathcal{O}_{\mathbb{P}^3} \rightarrow \mathcal{O}_{L}(2)) \otimes \mathcal{O}_{\mathbb{P}^3}(5)).$$
\begin{description}
\item[First, assume that $L \not \subset P$]  By Lemma \ref{pullback} we have:

\begin{equation*}
  \mathrm{Ext}^{1}(B,A)=
    \begin{cases}
     \mathrm{H}^{0}((\mathcal{O}_{P}(5) \oplus \mathcal{O}_{L \cap P}[-1]))= \C^{21} \oplus 0, &q \cup q'  \supset L \cap P,\\
  \mathrm{H}^{0}(\mathcal{I}_{L \cap P }(5))= \mathbb{C}^{20}, &  q \cup q' \not \supset L \cap P.
    \end{cases}
  \end{equation*}

\item [Second, assume that $L \subset P$]  Using Lemma \ref{pullback},  we have $\mathrm{H}^{0}(\iota_{P}^{*}(\mathcal{O}_{\mathbb{P}^3} \rightarrow \mathcal{O}_{L}(2)) \otimes \mathcal{O}_{\mathbb{P}^3}(5))= \mathrm{H}^{0}((\mathcal{O}_{P} \rightarrow \mathcal{O}_{L}(2))\otimes \mathcal{O}_{\mathbb{P}^3}(5)) \oplus \mathrm{H}^{0}(\mathcal{O}_{L}(1)\otimes \mathcal{O}_{\mathbb{P}^3}(5))= \mathrm{H}^{0}((\mathcal{O}_{P}(5) \rightarrow \mathcal{O}_{L}(7)))\oplus \mathrm{H}^{0}(\mathcal{O}_{L}(6))$. 

We have $\mathrm{H}^{0}((\mathcal{O}_{P}(5) \xrightarrow{s} \mathcal{O}_{L}(7)))=\mathrm{H}^{0}(\ker(s))=\mathrm{H}^{0}(\II_{L}(5))=\mathbb{C}^{15}$. Also, we have $\mathrm{H}^{0}(\mathcal{O}_{L}(6))=\mathbb{C}^{7}$. Therefore,  $\mathrm{H}^{0}(\iota_{P}^{*}(\mathcal{O}_{\mathbb{P}^3} \rightarrow \mathcal{O}_{L}(1)) \otimes \mathcal{I}_{Z'_{1}}(5))=\mathbb{C}^{15} \oplus  \mathbb{C}^{7}= \mathbb{C}^{22}$. 
\end{description}
Finally, for $\mathrm{Ext}^1(A,B)$ using Serre duality, we have
$$\mathrm{Ext}^{1}((\mathcal{O}_{\mathbb{P}^3}(-1) \rightarrow \mathcal{O}_{L}(1)),\mathcal{O}_{P}(-5))=\mathrm{H}^{1}(\iota_{P}^{*}(\mathcal{O}_{\mathbb{P}^3} \rightarrow \mathcal{O}_{L}(2)) \otimes \mathcal{O}_{\mathbb{P}^3}(1))^{\vee}.$$
\begin{description}

\item[\underline{1)  If $L \subset P$}]:
$$\mathrm{Ext}^{1}(A,B)=\mathrm{H}^{1}(\mathcal{O}_{P}(1) \xrightarrow{s} \mathcal{O}_{L}(3))^{\vee}\oplus \mathrm{H}^{1}(\mathcal{O}_{L}(2))^{\vee}= \mathrm{H}^{1}(\mathcal{O}_{P}(1) \xrightarrow{s} \mathcal{O}_{L}(3))^{\vee} \oplus 0.
$$
We have $\mathrm{H}^0(\OO_P(1) \xrightarrow{s} \OO_L(3))=\mathrm{H}^0(\ker(s))=\mathrm{H}^0(\II_{L}(1))=\C$. Now, taking the long exact sequence of  $\mathcal{O}_{L}(3)[-1] \rightarrow (\mathcal{O}_{P}(1) \xrightarrow{s} \mathcal{O}_{L}(3)) \rightarrow \mathcal{O}_{P}(1) $, gives

\begin{center}

\begin{tikzpicture}[descr/.style={fill=white,inner sep=1.5pt}]
        \matrix (m) [
            matrix of math nodes,
            row sep=1em,
            column sep=2.5em,
            text height=1.5ex, text depth=0.25ex
        ]
        {&  \mathrm{H}^{-1}(\mathcal{O}_{L}(3))=0 & \mathrm{H}^{0}(\mathcal{O}_{P}(1) \xrightarrow{s} \mathcal{O}_{L}(3)) &\mathrm{H}^{0}(\mathcal{O}_{P}(1))= \mathbb{C}^{3} \\
            &  \mathrm{H}^{0}(\mathcal{O}_{L}(3))= \mathbb{C}^{4} & \mathrm{H}^{1}(\mathcal{O}_{P}(1) \xrightarrow{s} \mathcal{O}_{L}(3)) &\mathrm{H}^{1}(\mathcal{O}_{P}(1))=0,\\
        };

        \path[overlay,->, font=\scriptsize,>=latex]
        
        (m-1-2) edge (m-1-3)
        (m-1-3) edge (m-1-4)
        (m-1-4) edge[out=355,in=175] node[descr,yshift=0.3ex] {} (m-2-2)
        (m-2-2) edge (m-2-3)
        (m-2-3) edge (m-2-4)
     ;
\end{tikzpicture}
\end{center}
which implies $\mathrm{H}^{1}(\mathcal{O}_{P}(1) \rightarrow \mathcal{O}_{L}(3))=\mathbb{C}^{2}$.
\item[\underline{2) If $L \not \subset P$}] using Lemma \ref{pullback}, we have
\begin{equation*}
  \mathrm{Ext}^{1}(A,B) =
    \begin{cases}
      \mathrm{H}^{1}(\mathcal{O}_{P}(1))^{\vee} \oplus  \mathrm{H}^{0}(\mathcal{O}_{L \cap P})^{\vee}=0 \oplus \mathbb{C}, &q \cup q'  \supset L \cap P,\\
   \mathrm{H}^{1}(\mathcal{I}_{L \cap P}(1))^{\vee}=0, & q \cup q' \not \supset L \cap P.

    \end{cases}
  \end{equation*}
\end{description}

\end{proof}

\begin{lem} ~\label{ext1pink}
For the wall $\langle\mathcal{O}_{\mathbb{P}^3}(-1),\iota_{P_{*}}\mathcal{I}_{Z_{6}}^{\vee}(-6)\rangle$, we have:

$$\mathrm{Ext}^{1}(\mathcal{O}_{\mathbb{P}^3}(-1),\mathcal{O}_{\mathbb{P}^3}(-1))=0,\quad
\mathrm{Ext}^{1}(\iota_{P_{*}}\mathcal{I}_{Z_{6}}^{\vee}(-6),\iota_{P_{*}}\mathcal{I}_{Z_{6}}^{\vee}(-6))= \mathbb{C}^{15},$$$$
\mathrm{Ext}^{1}(\iota_{P_{*}}\mathcal{I}_{Z_{6}}^{\vee}(-6), \mathcal{O}_{\mathbb{P}^3}(-1))=\mathbb{C}^{22},
$$
\begin{equation*}
    \mathrm{Ext}^{1}(\mathcal{O}_{\mathbb{P}^3}(-1),\iota_{P_{*}}\mathcal{I}_{Z_{6}}^{\vee}(-6))=
    \begin{cases}
      \mathbb{C}^{3}, &    \text{6 points on a line}\\
     \mathbb{C}^{2}, &    \text{5 points on a line}\\
     \mathbb{C}, &     \text{6 points on a conic}\\
     0, &    \text{generic  points}\\
    \end{cases}
 \end{equation*}

\end{lem}

\begin{proof}

 The first statement is obvious.  We notice that $\mathrm{Ext}^{1}(\iota_{P_{*}}\mathcal{I}_{Z_{6}}^{\vee}(-6),\iota_{P_{*}}\mathcal{I}_{Z_{6}}^{\vee}(-6))$ is a parameter space of six points in a plane which is of dimension 15 as claimed.

Now we compute $\mathrm{Ext}^{1}(B,A)$:
$$\mathrm{Ext}^{1}(\iota_{P_{*}}\mathcal{I}_{Z_{6}}^{\vee}(-6),\mathcal{O}_{\mathbb{P}^3}(-1))=\mathrm{Hom}(\mathcal{I}_{Z'_{6}}^{\vee}(-6),\iota_{P}^{*}(\mathcal{O}_{\mathbb{P}^3}))=\mathrm{H}^{0}(\mathcal{I}_{Z'_{6}}(6))= \mathbb{C}^{22}.$$

For the last part, using Serre duality we have
$$ \mathrm{Ext}^{1}(\mathcal{O}_{\mathbb{P}^3}(-1),\iota_{P_{*}}\mathcal{I}_{Z_{6}}^{\vee}(-6))= \mathrm{Ext}^{1}(\iota_{P}^{*}\mathcal{O}_{\mathbb{P}^3}(-1),\mathcal{I}_{Z_{6}}^{\vee}(-6))=\mathrm{H}^{1}(\mathcal{O}_{P}\otimes \mathcal{I}_{Z_{6}}(2))^{\vee}=\mathrm{H}^{1}(\mathcal{I}_{Z_{6}}(2))^{\vee}
$$
\begin{equation*}
    =
    \begin{cases}
      \mathbb{C}^{3}, &    \text{6 points on a line}\\
     \mathbb{C}^{2}, &    \text{5 points on a line}\\
     \mathbb{C}, &     \text{6 points on a conic}\\
     0, &    \text{generic  points}\\
    \end{cases}
  \end{equation*}
as $\dim \mathrm{H}^{1}(\mathcal{I}_{Z_{6}}(2))^{\vee}=\dim \mathrm{H}^{0}(\mathcal{I}_{Z_{6}}(2))^{\vee}$ by taking a long exact sequence of $\II_{Z_{6}}(2) \to \OO_P(2) \to \OO_{Z_{6}}$.

\end{proof}

\section{Ext computation of the walls on the Hilbert scheme side} \label{Ext computation of the walls on the Holbert scheme side}

In this section, we compute Ext groups for walls $\langle\mathcal{I}_{L_{2}}(-1), \mathcal{O}_{P}(-5)\rangle$, $\langle\mathcal{I}_{L_{1}}(-1), \mathcal{I}_{Z_1/P}(-5)\rangle$ and $\langle\mathcal{I}_{L}(-1), \mathcal{I}_{Z_2/P}(-5)\rangle$.
\begin{lem} ~\label{ext1purple1left}
Let $L_2=L \cup q \cup q'$, and $L \cap P=p$. For the wall $\langle\mathcal{I}_{L_{2}}(-1), \mathcal{O}_{P}(-5)\rangle$, we have:

\begin{equation*}
    \mathrm{Ext}^{1}(\mathcal{O}_{P}(-5),\mathcal{I}_{L_{2}}(-1))=
    \begin{cases}
    \C^{22}, &  L \subset P, q,q' \subset P \\
       \mathbb{C}^{21}, & L \subset P, q  \subset P, q' \not \subset P\\
    
       \mathbb{C}^{20}, & 
     \text{($L \not \subset P$) or ($L \subset P, q,q' \not \subset P$)}
     
    \end{cases}
  \end{equation*}

\end{lem}

\begin{proof} We only consider the most generic cases. Other cases are similar by using Lemma \ref{pullback}. Let $L \not \subset P$ and $q,q' \not \subset P $. Using Lemma \ref{pullback}, we have: 
  $$  \mathrm{Ext}^{1}(\mathcal{O}_{P}(-5),\mathcal{I}_{L_{2}}(-1))=\mathrm{Ext}^{1}(\mathcal{O}_{P}(-5),\iota^{!}\mathcal{I}_{L_{2}}(-1))=\mathrm{Ext}^{1}(\mathcal{O}_{P}(-5),\II_{p}[-1])$$$$=\mathrm{Hom}(\mathcal{O}_{P}(-5),\II_{p})=\C^{20}.$$
  
 Similarly, for the cases $q \subset P, q' \not \subset P$ and  $q, q' \subset P$, by using Lemma \ref{pullback}, we get $\C^{20}$.
 
 Now let $L \subset P$, and $q,q' \not \subset P$. Using Lemma \ref{pullback}, we have: 
  $$  \mathrm{Ext}^{1}(\mathcal{O}_{P}(-5),\mathcal{I}_{L_{2}}(-1))=\mathrm{Ext}^{1}(\mathcal{O}_{P}(-5),\iota^{!}\mathcal{I}_{L_{2}}(-1))=\mathrm{Ext}^{1}(\mathcal{O}_{P}(-5),\OO_P(-1) \oplus \OO_L(-1)[-1])$$$$=\mathrm{Hom}(\mathcal{O}_{P}(-4),\OO_{P}) \oplus \mathrm{Hom}(\mathcal{O}_{P}(-4),\OO_{L})=\C^{15}\oplus \C^5=\C^{20}.$$
  Similarly, for $q \subset P, q' \not \subset P$ and $q,q' \subset P$, we get $\C^{21}$ and $\C^{22}$, respectively.
\end{proof}

\begin{lem} ~\label{ext1purple2left}
For the wall $\langle\mathcal{I}_{L_{1}}(-1), \mathcal{I}_{Z_{1}/P}(-5)\rangle$, with $L_1=L \cup q$ and $p=L \cap P$, we have:

\begin{equation*}
    \mathrm{Ext}^{1}(\mathcal{I}_{Z_{1}/P}(-5),\mathcal{I}_{L_{1}}(-1))=
    \begin{cases}
    \C^{22}, &  L \not \subset P, p \neq q=Z_1 \subset P \\
       \mathbb{C}^{21}, & \text{($L \not \subset P, Z_1=p$) or ($L \subset P$)}\\
     
      \mathbb{C}^{20}, & L \not \subset P, Z_1 \neq p

    \end{cases}
  \end{equation*}

\end{lem}

\begin{proof} We only consider the most generic cases. Other cases are similar by using Lemma \ref{pullback}.
First let $ L \not \subset P, q \neq Z_1  \neq p$. We have
$$  \mathrm{Ext}^{1}(\mathcal{I}_{Z_{1}/P}(-5),\mathcal{I}_{L_{1}}(-1))=\mathrm{Ext}^{1}(\mathcal{I}_{Z_1}(-5),\iota^{!}\mathcal{I}_{L_{1}}(-1))=\mathrm{Ext}^{1}(\mathcal{I}_{Z_1}(-5),\II_{p}[-1])$$$$=\mathrm{Hom}(\mathcal{O}_{P}(-5),\II_{p})=\C^{20}.$$

Now, let $L \not \subset P, Z_1=p$ (the case $L \subset P$ is a degenerate case). We have
$$  \mathrm{Ext}^{1}(\mathcal{I}_{Z_{1}/P}(-5),\mathcal{I}_{L_{1}}(-1))=\mathrm{Ext}^{1}(\mathcal{I}_{Z_1}(-5),\iota^{!}\mathcal{I}_{L_{1}}(-1))=\mathrm{Ext}^{1}(\mathcal{I}_{Z_1}(-5),\II_{p}[-1])$$$$=\mathrm{Ext}^{1}(\mathcal{I}_{p}(-5),\II_{p}[-1])=\mathrm{Hom}(\mathcal{O}_{P}(-5),\OO_{P})=\C^{21}.$$

Finally, let $ L \not \subset P, p \neq q=Z_1 \subset P$. We have
$$  \mathrm{Ext}^{1}(\mathcal{I}_{Z_{1}/P}(-5),\mathcal{I}_{L_{1}}(-1))=\mathrm{Ext}^{1}(\mathcal{I}_{Z_1}(-5),\iota^{!}\mathcal{I}_{L_{1}}(-1))=\mathrm{Ext}^{1}(\mathcal{I}_{Z_1}(-5),(\II_{p \cup q}\oplus \OO_{p})[-1])$$$$=\mathrm{Hom}(\mathcal{I}_{Z_1}(-5),\II_{p \cup q})\oplus\mathrm{Hom}(\mathcal{I}_{Z_1}(-5),\OO_{p })=\mathrm{Hom}(\OO(-5),\II_{p })\oplus\mathrm{Hom}(\mathcal{I}_{Z_1}(-5),\OO_{p })$$$$ =\C^{20} \oplus \C^2=\C^{22}.$$

\end{proof}

\begin{lem} ~\label{ext1purple3left}Let $Z_2=Z \cup Z'$ and $p=L \cap P$. For the wall $\langle\mathcal{I}_{L}(-1), \mathcal{I}_{Z_{2}/P}(-5)\rangle$, we have:

\begin{equation*}
    \mathrm{Ext}^{1}(\mathcal{I}_{Z_{2}/P}(-5),\mathcal{I}_{L}(-1))=
    \begin{cases}
    \C^{22}, & Z_2 \subset L \subset P\\
       \mathbb{C}^{21}, & \text{($L \not \subset P$) and (either $Z$ or $Z'=p$),}\\
     
      \mathbb{C}^{20}, &  L \cap Z_2=\emptyset

    \end{cases}
  \end{equation*}

\end{lem}

\begin{proof} We only consider the most generic cases. Other cases are similar by using Lemma \ref{pullback}.
 First, let $L \cap Z_2=\emptyset$. We only need to look at the generic case when $L \not \subset P$:
$$  \mathrm{Ext}^{1}(\mathcal{I}_{Z_{2}/P}(-5),\mathcal{I}_{L}(-1))=\mathrm{Ext}^{1}(\mathcal{I}_{Z_2}(-5),\iota^{!}\mathcal{I}_{L}(-1))=\mathrm{Ext}^{1}(\mathcal{I}_{Z_2}(-5),\II_{p}[-1])$$$$=\mathrm{Hom}(\mathcal{I}_{Z_2}(-5),\II_{p})=\C^{20}.$$

Now, let $L \not \subset P$ and  $Z=p$ (the case $Z_2=p$ is similar):

$$  \mathrm{Ext}^{1}(\mathcal{I}_{Z_{2}/P}(-5),\mathcal{I}_{L}(-1))=\mathrm{Ext}^{1}(\mathcal{I}_{Z_2}(-5),\iota^{!}\mathcal{I}_{L}(-1))=\mathrm{Ext}^{1}(\mathcal{I}_{Z_2}(-5),\II_{p}[-1])$$$$=\mathrm{Hom}(\mathcal{I}_{Z \cup Z'}(-5),\II_{Z})=\mathrm{Hom}(\mathcal{I}_{ Z'}(-5),\OO_{P})=\C^{21}.$$

Finally, let $Z_2=Z \cup Z' \subset L \subset P$ and $Z,Z'$ are disjoint ($Z=Z'$ case is similar): 
$$  \mathrm{Ext}^{1}(\mathcal{I}_{Z_{2}/P}(-5),\mathcal{I}_{L}(-1))=\mathrm{Ext}^{1}(\mathcal{I}_{Z_2}(-5),\iota^{!}\mathcal{I}_{L}(-1))=\mathrm{Ext}^{1}(\mathcal{I}_{Z_2}(-5),(\OO_P(-1)\oplus \OO_L(-1))[-1])$$$$=\mathrm{Hom}(\mathcal{I}_{Z \cup Z'}(-4),\OO_{P}) \oplus \mathrm{Hom}(\mathcal{I}_{Z \cup Z'}(-4),\OO_{L}) =\C^{15} \oplus \C^{7}=\C^{22}.$$

\end{proof}

\section{Chambers of the stable pairs side} \label{chambers}

In this section, we describe the corresponding moduli spaces to the chambers close to the hyperbola from the right.  

First, we need the following Lemma which gives a condition under which we can realize what components can survive all the way to the large volume limit:

\begin{lem} \label{survival}
Suppose that $\HH^{0}$ of a general object in an irreducible component created by a wall is an ideal sheaf.  Then this irreducible component  survives all the way up to the moduli space of stable pairs.
\end{lem}
\begin{proof}
As stability is an open condition, we only need to show this for one object which is created after each wall. If an object $E$ is destabilized by a subobject $E'$, there is an injection $E' \hookrightarrow E$ in $\mathrm{Coh}^{ \alpha, \beta}(\mathbb{P}^{3})$ which induces an injection $\HH^0(E') \hookrightarrow \HH^0(E)$.  By assumption, $\HH^{0}$ of a general object created by the wall is an ideal sheaf $\II$.  
On the other hand, $\HH^0$ of the destabilizing subobjects are all of the form  $\OO_P(-i)$, by Theorem \ref{main}. But we have  $\Hom(\OO_P(-i), \II)=0$, therefore the induced map on $\HH^0$ is not injective, which is a contradiction. This implies the result.
\end{proof}
We also need the following Lemma:
\begin{lem} \label{lift}
Suppose that $E$ is a sheaf with $\mathrm{ch}_{0}(E)=1$ and $\mathrm{ch}_{1}(E)=0$, such that it fits into $\II_D(-1) \hookrightarrow E \twoheadrightarrow \OO_P(-k)$, for $\II_D$ an ideal sheaf of a subscheme $D$  transverse to the plane $P$, and $k$ a positive integer. Then $E$ is an ideal sheaf of a curve.
\end{lem}
\begin{proof}
First, we observe that 
$\iota_P^{*}(\II_D)=\mathcal{I}_{D\cap P/P}$. In order to show that $E$ is an ideal sheaf of a curve, we need to show that it is torsion-free. We know that any subsheaf of $\mathcal{O}_{P}(-k)$ contains a subsheaf of the form $\mathcal{O}_{P}(-i)$ for some $k < i$. It is enough to show that such $\mathcal{O}_{P}(-i)$ does not lift to a subsheaf of $E$. Using the identification $\mathrm{Ext}^1(\mathcal{O}_{P}(-i), \mathcal{I}_{D}(-1))=\mathrm{Ext}^1(\mathcal{O}_{P}(-i), \iota^{*}_{P}\mathcal{I}_{D}(-1)(1)[-1])=\mathrm{H}^0(\mathcal{I}_{D\cap P/P}(i))$, one can see that $ \mathrm{Ext}^{1}(\mathcal{O}_{P}(-k), \mathcal{I}_{D}(-1))  \hookrightarrow \mathrm{Ext}^{1}(\mathcal{O}_{P}(-i), \mathcal{I}_{D}(-1)) $ is injective. Hence $\mathcal{O}_{P}(-i)$ does not lift to a subsheaf of $E$.
\end{proof}

\begin{prop}  [{\cite[Proposition 3.14]{R1}}]
The moduli space for the first chamber below the first wall, which is given by all the extensions of two objects $\mathcal{O}_{\mathbb{P}^3}(-2), \mathcal{O}_{ Q}(-3)$, is empty.
\end{prop}

Let $\mathcal{N}_{1}$ be the first non-empty moduli space for the next chamber, which appears after crossing the smallest wall $\langle \mathcal{O}_{\mathbb{P}^3}(-2), \mathcal{O}_{ Q}(-3)\rangle$. In the following Proposition from \cite{R1}, we see that this moduli space gives a compactification of objects corresponding to $(2,3)$-complete intersection curves in $\mathbb{P}^{3}$:

\begin{prop}  [{\cite[Proposition 3.15]{R1}}] \label{N1}
The moduli space $\mathcal{N}_{1}$ is a $\mathbb{P}^{15}$-bundle over  $\mathbb{P}^{9}$. More precisely, the complement of (2,3)-complete intersections in $\mathcal{N}_{1}$ are parametrized by the pairs $(Q, C)$ where $Q=P \cup P'$ is a union of two planes and $C$ is a conic in one of the two planes. The associated objects are non-torsion free sheaves $E$  given by 
$\mathcal{O}_{P}(-4) \hookrightarrow E \twoheadrightarrow \mathcal{I}_{C_{2}}(-1),$ for a conic $C_{2}$. The generic element with $\HH^1$ of length $0$, is given by an ideal sheaf of (2,3)-complete intersection curves.
\end{prop}
\begin{proof}
The only remaining part which was not covered in \cite{R1}, is the description of the generic element as an ideal sheaf: this is easily obtained by noticing that $\HH^1(\OO_{\mathbb{P}^3}(-2))=0$, and then taking the long exact sequence of the defining sequence $ \mathcal{O}_{\mathbb{P}^3}(-2) \into E \onto \mathcal{O}_{ Q}(-3)$ and applying Lemma \ref{lift} to $D=P'$.
\end{proof}
Before describing the next chambers, we need the following Lemmas:

\begin{lem} [{\cite[Lemma 4.4]{GHS}}] \label{3.13}
Let $F \hookrightarrow E \twoheadrightarrow G$ be an exact sequence at a wall in Bridgeland stability with $E$ semistable to one side of the wall and $F, G$ distinct stable objects of the same (Bridgeland) slope. Then we have:
$$\mathrm{ext}^{1}(E,E) \leq \mathrm{ext}^{1}(F,F)+\mathrm{ext}^{1}(G,G)+ \mathrm{ext}^{1}(F,G)+ \mathrm{ext}^{1}(G,F)-1.$$
\end{lem}

\begin{lem}[{\cite{Moi}, \cite[Theorem 4.7]{GHS}}] \label{3.15}
Any birational morphism $f\colon X \rightarrow Y$ between smooth proper algebraic spaces of finite type over complex numbers s.t. the contracted locus $E$ is irreducible, and $f(E)$ is smooth, is the blow up of $Y$ in $f(E)$. 
\end{lem}

Now, we want to describe the moduli space $\mathcal{N}_{2}$ for the next chamber,  which comes after the wall $\langle\mathcal{I}_{C_{2}}(-1), \mathcal{O}_{P}(-4)\rangle$. Since $\mathrm{Ext}^{1}(\mathcal{I}_{C_{2}}(-1),\mathcal{O}_{P}(-4))=\mathbb{C}$ for all $\mathcal{I}_{C_{2}} \in {\mathcal{H}ilb^{2t+1}(\mathbb{P}^{3})}$ and all hyperplanes $P \in (\mathbb{P}^{3})^{\vee}$, for each choice of $C_2$ and $P$ there is a unique object in $\mathcal{N}_{1}$ destabilized at the wall $\langle\mathcal{I}_{C_{2}}(-1), \mathcal{O}_{P}(-4)\rangle$, identifying the destabilized locus with ${\mathcal{H}ilb^{2t+1}(\mathbb{P}^{3})} \times (\mathbb{P}^{3})^{\vee}$. 
\begin{prop}[{\cite[Proposition 3.18]{R1}}]  \label{N2}
The moduli space $\mathcal{N}_{2}$ for the next chamber is a blow up of $\mathcal{N}_{1}$ in the locus ${\mathcal{H}ilb^{2t+1}(\mathbb{P}^{3})} \times (\mathbb{P}^{3})^{\vee}$. The generic element in the exceptional locus with $\HH^1$ of length $0$, is given by an ideal sheaf $\II_{C_2 \cup C_4}$, for $C_4$ a plane quartic, and $C_2$ a conic not in the plane.
\end{prop}
\begin{proof}
Again the only remaining part which was not needed in \cite{R1}, is the description of the generic element as an ideal sheaf: this is easily obtained by noticing that $\mathrm{H}^1(\II_{C_2}(-1))=0$, and then taking the long exact sequence of the defining sequence $\mathcal{I}_{C_{2}}(-1) \into E \onto \mathcal{O}_{P}(-4)$ and applying Lemma \ref{lift} to $D=C_2$ in the general case.
\end{proof}

Let $\mathfrak{Fl}_2$ is the space  parametrizing flags $Z_{2} \subset P \subset \mathbb{P}^{3}$ where $P$ is a plane and $Z_{2}$ a zero-dimensional subscheme of length $2$. The next wall crossing was considered in {\cite[]{R1}}:

\begin{thm}[{\cite[Corollary 4.4]{R1}}] \label{N3}The moduli space $\mathcal{N}_{3}$ for the next chamber consists of two irreducible components: one is  $\widetilde {\mathcal{N}_{2}}$ which is birational to  $\mathcal{N}_{2}$; the other is a new component,  $\mathcal{N}'_{3}$ which is  a $\mathbb{P}^{17}$-bundle over $\mathbb{G}r(2,4) \times \mathfrak{Fl}_2 $. 
The latter generically parametrizes the union of a line and a plane quintic together with a choice of two points 
on the quintic. For any general element $E$ in the new component, $\HH^0(E)$ is 
an ideal sheaf, and $\HH^1(E)\neq 0$.

\end{thm}
\begin{proof}
 Again the only remaining part which was not needed in \cite{R1}, is the description of the generic element as an ideal sheaf: this is easily obtained by noticing that $\HH^1(\II_L(-1))=0$, and then taking the long exact sequence of the defining sequence $\II_L(-1) \into E \onto \iota_{P_{*}}\mathcal{I}_{Z_{2}}^{\vee}(-5)$ of any class $E$ in $\Ext^1(\iota_{P_{*}}\mathcal{I}_{Z_{2}}^{\vee}(-5),\II_L(-1))$, and applying Lemma \ref{lift} to $D=L$ in the general case. The claim $\HH^1$ being non-zero is obtained in a similar way and noticing that $\HH^1(\iota_{P_{*}}\mathcal{I}_{Z_{2}}^{\vee}(-5))\neq 0$. 
\end{proof}

Let $\mathcal{N}_{4}$ be the moduli space for the next chamber. Let $\mathcal{U}$ be the universal line over $\mathbb{G}r(2,4)$, and $\mathfrak{Fl}_l$ is the space  parametrizing flags $Z_{l} \subset P \subset \mathbb{P}^{3}$ where $P$ is a plane and $Z_{l}$ a zero-dimensional subscheme of length $l$.

\begin{prop}  \label{N4}
 The moduli space  $\mathcal{N}_{4}$ has four irreducible components: $\widetilde{\widetilde{\mathcal{N}_{2}}}$, $\widetilde{\mathcal{N}'_{3}}$, $\mathcal{N}'_{4}$ and $\mathcal{N}''_{4}$.  
 The first two are birational to their counterparts  in  $\mathcal{N}_{3}$. The component $\mathcal{N}'_{4}$ is 
 a $\mathbb{P}^{18}$-bundle over $\mathcal{U} \times \mathfrak{Fl}_1 $, and it generically parametrizes the union of  a line in $\P^3$ together with a choice of a point on it,  and  a plane quintic intersecting the line, together with a choice of a point  on it. The component $\mathcal{N}''_{4}$ is 
 a  $\mathbb{P}^{19}$-bundle over $\mathbb{G}r(2,4) \times \mathfrak{Fl}_1 $, and it  generically parametrizes disjoint unions of a  line in $\P^3$ and a plane quintic together with a choice of a point  on it. For both new components, $\HH^0$ of the generic element is given by an ideal sheaf,  and $\HH^1\neq 0$ of length $2$.

\end{prop}

\begin{proof}
The first part comes from Lemma \ref{survival} and Propositions  \ref{N2} and \ref{N3}. The fourth wall on the right side of the hyperbola is $\langle(\mathcal{O}_{\mathbb{P}^3}(-1) \xrightarrow{s} \mathcal{O}_{L}),\iota_{P_{*}}\mathcal{I}_{Z_{1}}^{\vee}(-5)\rangle$ (Theorem \ref{main}). Using Lemma \ref{ext1purple2} we obtain two  new components, $\mathcal{N}'_{4}$ and $\mathcal{N}''_{4}$ which are a $\mathbb{P}^{18}$-bundle  over $\mathcal{U} \times \mathfrak{Fl}_1 $ and a  $\mathbb{P}^{19}$-bundle  over $\mathbb{G}r(2,4)  \times \mathfrak{Fl}_1 $, respectively. We will show that the $\P^{20}$-bundle is contained in the closure of the $\P^{18}$-bundle.

To understand the objects in the new components more precisely, for any $E\in \mathrm{Ext}^1(B,A)$, we first want to understand its image in $\mathrm{Hom}(\HH^0(B), \HH^1(A))$ to see for which elements we get a non-zero map. Let $q$ be the zero locus of $s$. We notice that  we have the exact triangle $\HH^0(A)=\mathcal{I}_{L}(-1) \rightarrow A \rightarrow \HH^1(A)[-1]=\mathcal{O}_{q }[-1]$. Applying $\mathrm{RHom}(\HH^{0}(B)=\mathcal{O}_{P}(-5),-)$, we have:
$$0 \rightarrow \mathrm{Ext}^1(\mathcal{O}_{P}(-5), \mathcal{I}_L(-1))\rightarrow  \mathrm{Ext}^1(\mathcal{O}_{P}(-5), A) \rightarrow  \mathrm{Hom}(\mathcal{O}_{P}(-5),\HH^1(A)=\mathcal{O}_{q }) \rightarrow 0.$$
We consider the stratification of  $\mathcal{U} \times \mathfrak{Fl}_1 $, given by $\dim (\mathrm{Ext}^1(B,A))$, and consider a general object in each stratum:
\begin{description}
\item [(1) $L \not \subset P$, the zero  locus   of  $s$ is not   $ L \cap P$ and $ Z_1 \neq L \cap P$] 
In this case, using Lemma \ref{pullback}, we have
$$\mathrm{Ext}^1(\HH^0(B), A)=\mathrm{Ext}^{1}(\iota_{P_{*}}\mathcal{O}_{\mathbb{P}^3}(-5),\mathcal{O}_{\mathbb{P}^3}(-1) \rightarrow \mathcal{O}_{L})$$$$=\mathrm{Hom}(\mathcal{O}_{\mathbb{P}^3}(-5),\iota_{P}^{*}(\mathcal{O}_{\mathbb{P}^3} \rightarrow \mathcal{O}_{L}(1)))=\mathrm{H}^{0}(\iota_{P}^{*}(\mathcal{O}_{\mathbb{P}^3} \rightarrow \mathcal{O}_{L}(1)) \otimes \mathcal{O}_{\mathbb{P}^3}(5))= \mathrm{H}^0(\mathcal{I}_{L \cap P} (5)) =\C^{20},$$
 and $\mathrm{Hom}(\HH^0(B),\mathcal{O}_{q})=0$, and therefore any class in $\mathrm{Ext}^1(\HH^{0}(B), A)$ induces zero map in $\mathrm{Hom}(\HH^0(B),\HH^1(A))$. Moreover, we have 
 $$\Ext^1(B,A)=\mathrm{H}^{0}(\II_{(L \cap P)\cup Z_1}(5)) \hookrightarrow \Ext^{1}(\HH^{0}(B),A)=\mathrm{H}^0(\II_{L \cap P}(5)),$$
 and therefore the image is exactly given by quintics containing $L \cap P$ and $Z_1$.  Now, from

\begin{center}\label{connectingdiag}

\begin{tikzpicture}[descr/.style={fill=white,inner sep=1.5pt}]
        \matrix (m) [
            matrix of math nodes,
      row sep=1em,
            column sep=2.5em,
            text height=1.5ex, text depth=0.25ex
        ]
        {  &\HH^0(A) =\mathcal{I}_{L}(-1) &\HH^0(E) & \HH^0(B)=\mathcal{O}_{P}(-5)\\
             & \HH^1(A)=\mathcal{O}_{q } & \HH^1(E) & \HH^1(B)=\mathcal{O}_{Z_{1} },\\
        };

        \path[overlay,->, font=\scriptsize,>=latex]
        
        (m-1-2) edge (m-1-3)
        (m-1-3) edge (m-1-4)
        (m-1-4) edge[out=355,in=175] node[descr,yshift=0.3ex] {} (m-2-2)
        (m-2-2) edge (m-2-3)
        (m-2-3) edge (m-2-4)
      ;
\end{tikzpicture}

\end{center}
as the connecting map is zero, $\HH^1(E)$ have length 2.  Moreover, the top row is a short exact sequence $0 \to \II_L(-1) \hookrightarrow \HH^{0}(E) \twoheadrightarrow \OO_P(-5) \to 0$. The induced extension class 
$$ \mathrm{Ext}^1(\mathcal{O}_{P}(-5), \II_L(-1))=\mathrm{Ext}^1(\mathcal{O}_{P}(-5), i^{*}_{P} \II_L(-1)(1)[-1])=\mathrm{H}^0(\mathcal{I}_{L \cap P}(5))$$
 corresponds to quintics containing the intersection point.  Using Lemma \ref{lift}, 
 this means that for a generic object in $\Ext^1(B,A)$, the sheaf  $\HH^{0}(E)$ is ideal sheaf of  the  union of a line in $\P^3$ and a plane quintic intersecting in $L \cap P$. This together with the description of $\HH^{1}(E)$  implies the result for $\NN'_{4}$. 
\item [(2) $L \not \subset P$ and zero  locus   of  $s$  is   $L \cap P$, but $\neq Z_1$]  Using Lemma \ref{pullback}, we have

\begin{tikzcd}
                                                                                               & {\Ext^1(\HH^0(B),\HH^0(A))=\mathrm{H}^0(\II_q(5))} \arrow[d, "b"]            \\
{\mathrm{Ext}^1(B,A)=\mathrm{H}^{0}(\II_{Z_1}(5))} \arrow[r, "a"] & {\Ext^1(\HH^0(B),A)=\mathrm{H}^{0}(\OO_P(5))} \arrow[d, ""] \\
                                                                                               & {\Hom(\HH^0(B),\HH^1(A))=\mathrm{H}^0(\OO_q)=\C.}                                              
\end{tikzcd}

As $\im(a) \not \subset \im(b)$, the general element in $\mathrm{Ext}^1(B,A)$ induces a non-zero morphism $\HH^0(B) \to \HH^1(A)$, and so the connecting map in the above diagram is non-zero, which means $\HH^1(E)$ has length 1.  From the above diagram we have $\im(\HH^{0}(E) \to \HH^0(B)=\OO_P(-5))=\ker(\HH^0(B)=\OO_P(-5) \to \HH^{1}(A)=\OO_q)=\II_q(-5)$, which induces $0 \to \II_L(-1) \hookrightarrow \HH^{0}(E) \twoheadrightarrow \II_q(-5) \to 0$. This means that the extension class  
$$ \mathrm{Ext}^1(\II_q(-5), \II_L(-1))=\mathrm{Ext}^1(\iota_{P_{*}}\II_{L \cap P}(-5), \II_L(-1))$$$$=\mathrm{Ext}^1(\II_{L \cap P}(-5), i^{*}_{P} \II_L(-1)(1)[-1])=\mathrm{Hom}(\II_{L \cap P}(-5), \mathcal{I}_{L \cap P})=\mathrm{H}^{0}(\OO_P(5))$$
 corresponds to quintics not necessarily containing the intersection point.

 This means for a generic object in $\Ext^1(B,A)$, the sheaf   $\HH^{0}(E)$ is an ideal sheaf of the disjoint union of a line in $\P^3$ and a plane quintic. This together with the description of $\HH^{1}(E)$  completes the result for $\NN''_{4}$. 

\item [(3) $L \not \subset P$ and zero  locus   of  $s$  is   $L \cap P=Z_1$] Consider the diagram (by Lemma \ref{ext1purple2})

\begin{tikzcd}
                                                                                               & {\Ext^1(\HH^0(B),\HH^0(A))=\mathrm{H}^0(\II_q(5))} \arrow[d, "b"]            \\
{
\mathrm{H}^{0}(\II_q(5) \oplus (\II_q \otimes \OO_q[-1]))} \arrow[r, "a"] & {\Ext^1(\HH^0(B),A)=\mathrm{H}^{0}(\OO_P(5)\oplus \OO_q[-1])} \arrow[d, "c"] \\
                                                                                               & {\Hom(\HH^0(B),\HH^1(A))=\mathrm{H}^0(\OO_q)=\C.}                                              
\end{tikzcd}

As $\im(a)=\im(b)$, the composition $c \circ a$ is zero. Therefore,  general $E \in \Ext^1(B,A)$ induces the zero connecting map in the above diagram \ref{connectingdiag}, i.e., it gives the following short exact sequences
$$\HH^0(A) =\mathcal{I}_{L}(-1) \hookrightarrow \HH^0(E) \twoheadrightarrow \HH^0(B)=\mathcal{O}_{P}(-5)$$
and
$$ \HH^1(A)=\mathcal{O}_{q }\hookrightarrow \HH^1(E) \twoheadrightarrow \HH^1(B)=\mathcal{O}_{q }.$$
Therefore,  the plane quintic is smooth, and using Lemma \ref{lift}, for $E \in \Ext^1(B,A)$ general, we have $\HH^0(E)=\II_{L \cup C_5}$ corresponds to the connected union of the quintic with $L$ intersecting in a node, and so the curve $L \cup C_5$ is Gorenstein. Thus, we can first apply  Lemma \ref{charPT} to realize $E$ as a stable pair, and then apply {\cite[Proposition B.5]{PT2009}} to identify the stable pair $E$ with $\HH^0(E)=\II_{L \cup C_5}$, with a length $2$ subschemes of $L \cup C_5$.   But from the second short exact sequence above, $\HH^1(E)$ is  supported at $q$, and therefore this length $2$ subscheme is supported at $q$.  On the other hand, any such subscheme is in the closure of the locus of $\HH ilb^{[2]}(L \cup C_5)$ corresponding to one point on $L$ and one point on $C_5$.  This implies that every object in the $\P^{20}$-bundle is in the closure of the $\P^{18}$-bundle over $\UU \times \mathfrak{Fl}_1$.
\end{description}
\end{proof}

Let $\mathcal{N}_{5}$, be the moduli space for the next chamber. Let $C_4 \subset P$ be a plane quartic and $\L$ a thickening of a line $L \subset P$. The other notations as defined right before Proposition \ref{N4}:
\begin{prop}  \label{N5}
The moduli space $\mathcal{N}_{5}$ has seven irreducible components: $\widetilde{\widetilde{\widetilde{\mathcal{N}_{2}}}}$, $\widetilde{\widetilde{ \mathcal{N}'_{3}}}$, $\widetilde{ \mathcal{N}'_{4}}$, $\widetilde{ \mathcal{N}''_{4}}$, $\mathcal{N}'_{5}$, $\mathcal{N}''_{5}$ and $\mathcal{N}'''_{5}$.  The first four are birational to their counterparts in  $\mathcal{N}_{4}$.  The component $\mathcal{N}'_{5}$ is a $\mathbb{P}^{19}$-bundle over $(\mathcal{U} \times_{\mathbb{G}r(2,4)} \mathcal{U}) \times (\mathbb{P}^{3})^{*}$, and it generically parametrizes the union of a line in $\P^3$ together with a choice of two points on it and a plane quintic intersecting the line. The component  $\mathcal{N}''_{5}$ is a $\mathbb{P}^{20}$-bundle over $\mathcal{U} \times (\mathbb{P}^{3})^{*}$, and it generically parametrizes the disjoint unions of a plane quintic and  a line in $\P^3$ together with a choice of a point on it. The component $\mathcal{N}'''_{5}$ is a $\mathbb{P}^{21}$-bundle over $\mathfrak{Fl}_2$, and it generically parametrizes the union of a plane quartic with a thickening of a line in the plane.  For $\mathcal{N}'_{5}$  and $\mathcal{N}''_{5}$,   any generic element has $\HH^0$ an ideal sheaf, and $\HH^1$ non-zero. In $\mathcal{N}'''_{5}$, any generic element is of the form $\II_{\L \cup C_4}$.

\end{prop}

\begin{proof}
The first part comes from Lemmas \ref{survival}  and Propositions  \ref{N2}, \ref{N3} and \ref{N4}. The fifth wall on the right side of the hyperbola is $\langle(\mathcal{O}_{\mathbb{P}^3}(-1) \xrightarrow{s} \mathcal{O}_{L}(1)), \mathcal{O}_{P}(-5)\rangle$ (Theorem \ref{main}). Using Lemma \ref{ext1purple3}, we get three  new components $\mathcal{N}'_{5}$, $\mathcal{N}''_{5}$ and $\mathcal{N}'''_{5}$  which are  a $\mathbb{P}^{19}$-bundle  over $(\mathcal{U} \times_{\mathbb{G}r(2,4)} \mathcal{U}) \times (\mathbb{P}^{3})^{*}$, a $\mathbb{P}^{20}$-bundle  over $\mathcal{U} \times (\mathbb{P}^{3})^{*}$, and a $\mathbb{P}^{21}$-bundle over $\mathfrak{Fl}_2$, respectively. 

 To understand the objects in the new components more precisely, for any $E\in \mathrm{Ext}^1(B,A)$, we want to understand its image in $\mathrm{Hom}(B, \HH^1(A))$ to see for which elements we get a non-zero map. Let $q \cup q'$ be the zero locus of $s$. We notice  we have the short exact sequence $\HH^0(A)=\mathcal{I}_{L}(-1) \rightarrow A \rightarrow \HH^1(A)[-1]=\mathcal{O}_{q \cup q'}[-1]$.  Taking $\mathrm{RHom}(\HH^0(B)=B,-)$ of this, we have
$$0 \rightarrow \mathrm{Ext}^1(B, \II_L(-1))=\mathbb{C}^{20} \rightarrow  \mathrm{Ext}^1(B, A) \rightarrow  \mathrm{Hom}(B,\HH^1(A)=\mathcal{O}_{q \cup q'}) \rightarrow 0.$$
Now, there are three cases:
\begin{description}

\item[(1) $L \not \subset P$ and  zero  locus   of  $s$  does not contain  $L \cap P$] In this case we have $\mathrm{Ext}^1(B, A)$
$=\mathrm{H}^0(\mathcal{I}_{L \cap P}(5))=\mathbb{C}^{20}$ (Lemma \ref{ext1purple3}) and $\mathrm{Hom}(B,\HH^1(A))=0$, and therefore any class in $\mathrm{Ext}^1(B, A)$ induces zero map in $\mathrm{Hom}(B,\HH^1(A))$. Now from 

\begin{center}

\begin{tikzpicture}[descr/.style={fill=white,inner sep=1.5pt}]
        \matrix (m) [
            matrix of math nodes,
            row sep=1em,
            column sep=2.5em,
            text height=1.5ex, text depth=0.25ex
        ]
        {   &\HH^0(A) =\mathcal{I}_{L}(-1) &\HH^0(E) & \HH^0(B)=\mathcal{O}_{P}(-5)\\
             & \HH^1(A)=\mathcal{O}_{q \cup q'} & \HH^1(E) & \HH^1(B)=0,\\
        };

        \path[overlay,->, font=\scriptsize,>=latex]
       
        (m-1-2) edge (m-1-3)
        (m-1-3) edge (m-1-4)
        (m-1-4) edge[out=355,in=175] node[descr,yshift=0.3ex] {} (m-2-2)
        (m-2-2) edge (m-2-3)
        (m-2-3) edge (m-2-4)
      ;
\end{tikzpicture}
\end{center}
as the connecting map is zero, $\HH^1(E)$ have length 2, and the top row is a short exact sequence $0 \to \II_L(-1) \hookrightarrow \HH^{0}(E) \twoheadrightarrow \OO_P(-5) \to 0$. The extension class
$$ \mathrm{Ext}^1(\mathcal{O}_{P}(-5), \II_L(-1))=\mathrm{H}^0(\mathcal{I}_{L \cap P}(5))$$
 corresponds to quintics containing the intersection point. Using Lemma \ref{lift}, 
 this means for a generic choice in $\Ext^1(B,A)$, the sheaf $\HH^{0}(E)$ is the ideal sheaf of the union of a line in $\P^3$ and a plane quintic intersecting the line. This together with the description of $\HH^{1}(E)$  gives the result about 
 $\NN'_{5}$. 

\item[(2) $L \not \subset P$ and  zero  locus   of  $s$  contains  $L \cap P$] In this case, we have 
$$\Ext^1(B,\HH^0(A))=\mathrm{H}^0(\II_{L \cap P} (5)) \hookrightarrow \Ext^1(B,A)=\mathrm{H}^{0}(\OO_P(5)) \rightarrow \Hom(B,\HH^1(A))=\C,$$
and so any general class in $\mathrm{Ext}^1(B, A)$ induces a non-zero map in $\mathrm{Hom}(B,\HH^1(A))$. This time as the connecting map in the above diagram is non-zero, $\HH^1(E)$ have length 1.  Similarly as in the previous Lemma, we have $\im(\HH^{0}(E) \to B=\OO_P(-5))=\ker(B=\OO_P(-5) \to \HH^{1}(A)=\OO_{q \cup q'})=\iota_{P_{*}}\II_{L \cap P}(-5)$. Therefore the diagram induces a short exact sequence $0 \to \II_L(-1) \hookrightarrow \HH^{0}(E) \twoheadrightarrow \iota_{P_{*}}\II_{L \cap P}(-5) \to 0$. This means that the extension class  
$$ \mathrm{Ext}^1(\iota_{P_{*}}\II_{L \cap P}(-5), \II_L(-1))=\mathrm{Ext}^1(\II_{L \cap P}(-5), i^{*}_{P} \II_L(-1)(1)[-1])$$$$=\mathrm{Hom}(\II_{L \cap P}(-5), \mathcal{I}_{L \cap P})=\mathrm{H}^{0}(\OO_P(5))$$
 corresponds to quintics not necessarily containing the intersection point. 
 This means for a generic class in $\Ext^1(B,A)$, the sheaf $\HH^{0}(E)$ is the ideal sheaf of the disjoint union of a line in $\P^3$ and a plane quintic. This together with the description of $\HH^{1}(E)$  implies the result 
 for $\NN''_{5}$. 

\item[(3) $L \subset P$] Recall that from \ref{ext1purple3}, we have $\Ext^1(B,A)= \mathrm{H}^0(\OO_L(6))\oplus \mathrm{H}^0(\II_L(5))=\C^6 \oplus \C^{15}=\C^{22} $. Therefore in this case, we have a $\P^{21}$-bundle over a $3+4=7$-dimensional (3 for the choice of a plane and 3 for two points in the plane) locus, which gives a stratum of dimension $28$, and therefore it corresponds to a new component, $\NN'''_5$. 

Notice that we have $\Ext^2(B,\HH^0(A))=0$; thus we have

$$\Ext^1(B,A) \xrightarrowdbl{\gamma} \Hom(B,\HH^1(A)).$$

Surjectivity of $\gamma$ implies that a general class in $\Ext^1(B,A)$ induces a surjective connecting map $B \to \HH^1(A)$. This means that $\HH^1(E)=0$, for any general element $E$ in the component. Therefore we have $\ker (\text{connecting map})=\ker(\OO_P(-5) \twoheadrightarrow \OO_{q \cup q'})=\II_{q \cup q'/P}(-5)$, and thus any general element $E$ in the component fits into
$$\II_L(-1) \hookrightarrow E \twoheadrightarrow \II_{q \cup q'/P}(-5).$$


Now, let $\L$ be the double line obtained by thickening   $L$, with tangent direction of infinitesimal thickening contained in the plane at $(L \cap C_4) \cup q \cup q'$. From $\L \cup C_4 \subset \L \cup P$, we get $$\II_{L}(-1)=\II_{\L \cup P} \hookrightarrow \II_{\L \cup C_4} \twoheadrightarrow \II_{(\L \cup C_4)\cap P/P}=\II_{L \cup C_4 \cup q \cup q'/P}=\II_{q \cup q'/P}(-5).$$
On the other hand, considering the composition $E=\II_{\L \cup C_4} \twoheadrightarrow \II_{q \cup q'/P}(-5) \hookrightarrow \OO_P(-5)$ gives the exact triangle
$$\big(E \to \OO_P(-5)\big) \to E \to \OO_P(-5).$$
But $\big(E \to \OO_P(-5)\big)  \cong \big(\OO_{\mathbb{P}^3}(-1) \to \OO_L(1)\big)=A$: Let $K[1]$ be the cone of the composition $E \to  \OO_P(-5)$. From the octahedral axiom, we get an exact triangle $\II_L(-1) \to K \to \mathcal{O}_{q \cup q'}[-1]$. But we have $Ext^1(\mathcal{O}_{q \cup q'}[-1], \II_L(-1))=0$, and therefore we get $K=\mathcal{O}_{q \cup q'}[-1] \oplus \II_L(-1) $. Then the original exact triangle we started with implies $[E \to \OO_P(-5)]\cong \mathcal{O}_{q \cup q'}[-1] \oplus \II_L(-1) $. Therefore $\HH^{0}([E \to \OO_P(-5)])
= \II_L(-1)=\HH^0(A)$, and $\HH^{1}([E \to \OO_P(-5)])
=\mathcal{O}_{q \cup q'}=\HH^1(A)$.

Therefore our $E$ also arises via a class in $\Ext^1(B,A)$. 
This means for a generic class in $\Ext^1(B,A)$, the sheaf $E$ is the ideal sheaf of  a plane quartic with a thickening of the line $L$. This completes the result about $\NN'''_{5}$. 
\end{description}
\end{proof}

Let $\mathcal{N}_{6}$, be the moduli space for the next chamber. The other notations as defined right before Proposition \ref{N4}:

\begin{prop}  \label{N6}
The moduli space $\mathcal{N}_{6}$ has eight irreducible components: $\widetilde{\widetilde{\widetilde{\widetilde{\mathcal{N}_{2}}}}}$, $\widetilde {\widetilde{\widetilde{\mathcal{N}'_{3}}}}$, $\widetilde{\widetilde{\mathcal{N}'_{4}}}$, $\widetilde{\widetilde{\mathcal{N}''_{4}}}$, $\widetilde{\mathcal{N}'_{5}}$, $\widetilde{\mathcal{N}''_{5}}$, $\widetilde{\mathcal{N}'''_{5}}$, and $\mathcal{N}'_{6}$.  
The first seven are birational to their counterparts  in  $\mathcal{N}_{5}$.  The component $\mathcal{N}'_{6}$ is a $\mathbb{P}^{21}$-bundle over $\mathfrak{Fl}_6$, and it generically parametrizes plane degree 6 curves together with a choice of 6 points on it.  For any general element $E$ in the new component, $\HH^0(E)$ is 
an ideal sheaf $\II_{C_6}$, where $C_6$ is a plane sextic curve, and $\HH^1 \neq 0$. 

\end{prop}

\begin{proof}
The first part comes from Lemma \ref{survival},  and Propositions  \ref{N2}, \ref{N3}, \ref{N4} and \ref{N5}. For the new component, we notice that the sixth wall on the right side of the hyperbola is given by $\langle\mathcal{O}_{\mathbb{P}^3}(-1),\iota_{P_{*}}\mathcal{I}_{Z_{6}}^{\vee}(-6)\rangle$ (Theorem \ref{main}). Using Lemma \ref{ext1pink}, the new component, $\mathcal{N}'_{6}$, is a $\mathbb{P}^{21}$-bundle over $\mathfrak{Fl}_6$  generically parametrizes plane degree 6 curves together with a choice of 6 points on it. Description of the generic element as an ideal sheaf is easily obtained by noticing that $\HH^1(\OO_{\mathbb{P}^3}(-1))=0$, and then taking the long exact sequence of the defining sequence $\OO_{\mathbb{P}^3}(-1) \into E \onto \iota_{P_{*}}\mathcal{I}_{Z_{6}}^{\vee}(-5)$ of any class $E$ in $\Ext^1(B,A)$, and applying Lemma \ref{lift} to $D=\emptyset$. The claim $\HH^1$ being non-zero is obtained in a similar way and noticing that $\HH^1(\iota_{P_{*}}\mathcal{I}_{Z_{6}}^{\vee}(-6))\neq 0$ of length $6$.

\end{proof}

\section{Chambers of the Hilbert scheme side} \label{chambersHS}

In this section, we look at the loci created by walls $\langle\mathcal{I}_{L_{2}}(-1), \mathcal{O}_{P}(-5)\rangle$, $\langle\mathcal{I}_{L_{1}}(-1), \mathcal{I}_{Z_1/P}(-5)\rangle$ and $\langle\mathcal{I}_{L}(-1), \mathcal{I}_{Z_2/P}(-5)\rangle$. We denote by $\mathfrak{PL}_i$, the parametrizing space of a plane, a line, and $i$ floating points. Note that in the proofs below, as we use corresponding statement in Section \ref{Ext computation of the walls on the Holbert scheme side} regarding ext-calculation (with precise reference in the statements below), we end up with expressions like $\mathrm{Hom}(\mathcal{I}_{W_i/P}(-m),\II_{Z_i})$, where $Z_i,W_i$ are 0-dimensional subschemes. But  $\mathrm{Hom}(\mathcal{I}_{W_i/P}(-m),\II_{Z_i})$ is the same as $\mathrm{H}^0(I_{q_i/P}(m))$, for some 0-dimensional subscheme $q_i$, which parametrizes degree $m$ plane curves together with some extra points whose length and intersection information come from the proof of corresponding lemmas in Section  \ref{Ext computation of the walls on the Holbert scheme side}.

\begin{prop}
 \label{prop: H2component}
The locus $\mathscr{H}_{2}$ created by the wall $\langle\mathcal{I}_{L_{2}}(-1), \mathcal{O}_{P}(-5)\rangle$ is a $\P^{19}$-bundle over $\mathfrak{PL}_2$, of dimension $32$ and generically parametrizes the union of a line in $\P^3$ and a plane quintic intersecting at one point, together with $2$ floating points.
\end{prop}

\begin{proof}
Recall that $L_2$ is the union of $L$ and two floating points. By the proof of Lemma \ref{ext1purple1left}, when $L \not \subset P$, then for the generic case we have $$\mathrm{Ext}^{1}(\mathcal{O}_{P}(-5),\mathcal{I}_{L_{2}}(-1))=\mathrm{Hom}(\mathcal{O}_{P}(-5),\II_{p})=\C^{20}.$$ As this is just $\mathrm{H}^0(\II_p(5))$, it corresponds to a plane quintic in $P$ containing $p=L \cap P$; this induces a $\P^{19}$-bundle over a locus which parametrizes a plane, a line and two floating points. Hence the component has dimension $19+(3+4+2 \times 3)=32$. Similarly, the non-generic cases give loci of dimensions $28, 29$ in the closure of the generic case.  

 On the other hand, when $L \subset P$ and $q,q' \not \subset P$, by the proof of Lemma \ref{ext1purple1left} we have $\mathrm{Ext}^{1}(\mathcal{O}_{P}(-5),\mathcal{I}_{L_{2}}(-1))=\mathrm{Hom}(\mathcal{O}_{P}(-4),\OO_P) \oplus \mathrm{Hom}(\mathcal{O}_{P}(-4),\OO_L)=\C^{15} \oplus \C^5=\C^{20}.$ This corresponds to a plane quartic and a
double line in the plane together with two floating points outside
the plane. This gives a locus of dimension $19+(3+2+6)=30$ in the closure of the generic case.

Similarly, for $q \subset P, q' \not \subset P$ and $q,q' \subset P$, as in Lemma \ref{ext1purple1left}, we get $\C^{21}$ and $\C^{22}$, respectively. The first one  corresponds to a plane quartic and a
double line in the plane together with one point on the line and one floating point outside the plane, which gives a locus of dimension $20+(3+2+2+3)=30$. The second one corresponds to a plane quartic and a double line in the plane together with two points on the line, which gives a locus of dimension $21+(3+2+2+2)=30$. In both cases as we have plane curves, we have local complete intersections, and hence by using {\cite[Theorem 1.3]{CN12}}, we can detach the point from the curve and deform it to two floating points outside the plane; hence both loci are in the closure of the generic case, i.e., in $\mathscr{H}_{2}$.
\end{proof}

\begin{prop}
 \label{prop: H1component}
 The locus $\mathscr{H}_{1}$ created by the wall $\langle\mathcal{I}_{L_{1}}(-1), \mathcal{I}_{Z_1/P}(-5)\rangle$ is a $\P^{20}$-bundle over $\mathfrak{PL}_1$, of dimension $30$ and generically parametrizes the disjoint union of a line in  $\P^3$ and a plane quintic together with $1$ floating point.
\end{prop}
\begin{proof}
By Lemma \ref{ext1purple2left}, there are three possibilities for  $\mathrm{Ext}^{1}(\mathcal{I}_{Z_1/P}(-5),\mathcal{I}_{L_{1}}(-1))$, which we only consider generic case of each:

1)  $\mathrm{Ext}^{1}(\mathcal{I}_{Z_1/P}(-5),\mathcal{I}_{L_{1}}(-1))=\C^{22}$: in this case, as in the proof of Lemma \ref{ext1purple2left},  $$\mathrm{Ext}^{1}(\mathcal{I}_{Z_1/P}(-5),\mathcal{I}_{L_{1}}(-1))=\mathrm{Hom}(\OO_P(-5),\II_{p })\oplus\mathrm{Hom}(\mathcal{I}_{Z_1}(-5),\OO_{p } \oplus \OO_p) =\C^{20} \oplus \C \oplus \C=\C^{22}.$$
 This corresponds to a plane quintic in $P$ containing $p=L \cap P$; this induces a $\P^{21}$-bundle over a locus which parametrizes a plane, a line and a double point in the plane. Hence the component is contained in the closure of $\mathscr{H}_{2}$, and has dimension $21+(3+4+1\times2)=30$.

2)  $\mathrm{Ext}^{1}(\mathcal{I}_{Z_1/P}(-5),\mathcal{I}_{L_{1}}(-1))=\C^{21}$: in this case, as in the proof of Lemma \ref{ext1purple2left},  $$\mathrm{Ext}^{1}(\mathcal{I}_{Z_1/P}(-5),\mathcal{I}_{L_{1}}(-1))=\mathrm{H}^{0}(\OO_P(-5))=\C^{21}.$$
 This corresponds to a plane quintic in $P$ not containing $p=L \cap P$; this induces a $\P^{20}$-bundle over a locus which paramerrizes a plane, a line and a floating point. Hence this is a new component $\mathscr{H}_{1}$ which is not contained in the closure of $\mathscr{H}_{2}$, and has dimension $20+(3+4+3)=30$.
 
 3)  $\mathrm{Ext}^{1}(\mathcal{I}_{Z_1/P}(-5),\mathcal{I}_{L_1}(-1))=\C^{20}$: in this case, as in the proof of Lemma \ref{ext1purple2left},  $$\mathrm{Ext}^{1}(\mathcal{I}_{Z_1/P}(-5),\mathcal{I}_{L_{1}}(-1))=\mathrm{Hom}(\mathcal{I}_{Z_1}(-5),\II_{p })= \C^{20}.$$
 This corresponds to a plane quintic in $P$ containing $p=L \cap P$ and not containing $Z_1$; this induces a $\P^{19}$-bundle over a locus which paramerrizes a plane, a line a double point in the plane and another floating point. Hence the component is contained in the closure of $\mathscr{H}_{2}$, and has dimension $19+(3+4+2+3)=31$.

\end{proof}

In the statement of the following proposition, we refer to the part (3) of Proposition \ref{N5}, for the precise meaning of thickening.
\begin{prop}
 \label{prop: HCMcomponent}
The locus $\mathscr{H}'_{CM}$ created by the wall $\langle\mathcal{I}_{L}(-1), \mathcal{I}_{Z_2/P}(-5)\rangle$ is of dimension $28$ and generically  parametrizes  the  union  of  a  plane  quartic  with  a  thickening  of a line in the plane, which is the same as $\mathcal{N}'''_{5}$ (Proposition \ref{N5}).
\end{prop}

\begin{proof}
By Lemma \ref{ext1purple3left}, there are three possibilities for  $\mathrm{Ext}^{1}(\mathcal{I}_{Z_2/P}(-5),\mathcal{I}_{L}(-1))$, which we only consider generic case of each:

1)  $\mathrm{Ext}^{1}(\mathcal{I}_{Z_2/P}(-5),\mathcal{I}_{L}(-1))=\C^{22}$: in this case, as in the proof of Lemma \ref{ext1purple3left},   $$\mathrm{Ext}^{1}(\mathcal{I}_{Z_2/P}(-5),\mathcal{I}_{L}(-1))=\mathrm{Hom}(\mathcal{I}_{Z \cup Z'}(-4),\OO_{P}) \oplus \mathrm{Hom}(\mathcal{I}_{Z \cup Z'}(-4),\OO_{L}).$$
 This  generically  parametrizes  the  union  of  a  plane  quartic  with  a  thickening  of a line in the plane, which is the same as $\mathcal{N}'''_{5}$  in Proposition \ref{N5}, which was proved to be irreducible and has dimension $21+(3+2+2)=28$.

2)  $\mathrm{Ext}^{1}(\mathcal{I}_{Z_2/P}(-5),\mathcal{I}_{L}(-1))$: in this case, as in the proof of Lemma \ref{ext1purple3left},  $$\mathrm{Ext}^{1}(\mathcal{I}_{Z_2/P}(-5),\mathcal{I}_{L}(-1))=\mathrm{Hom}(\mathcal{I}_{ Z'}(-5),\OO_P)=\C^{21}.$$
 This corresponds to a plane quintic in $P$ not containing $p=L \cap P$; this induces a $\P^{20}$-bundle over a locus which paramerrizes a plane, a line and a point in the intersection and another point in the plane. Hence this is contained in the closure of $\mathscr{H}_{1}$, and has dimension $20+(3+4+2)=29$.
 
 3)  $\mathrm{Ext}^{1}(\mathcal{I}_{Z_2/P}(-5),\mathcal{I}_{L}(-1))=\C^{20}$: in this case, as in the proof of Lemma \ref{ext1purple3left},  $$\mathrm{Ext}^{1}(\mathcal{I}_{Z_2/P}(-5),\mathcal{I}_{L}(-1))=\mathrm{Hom}(\mathcal{I}_{Z_2}(-5),\II_{p})=\C^{20}.$$
 This corresponds to a plane quintic in $P$ containing $p=L \cap P$ and not containing $Z_2$; this induces a $\P^{19}$-bundle over a locus which parametrizes a plane, a line, and two points in the plane. Hence the component is contained in the closure of $\mathscr{H}_{2}$, and has dimension $19+(3+4+2 \times 2)=30$.
\end{proof}

\begin{cor} \label{Corollary: irred}
$\mathscr{H}_{2}$ is irreducible.
\end{cor}

\begin{proof}
We note that Propositions \ref{prop: H1component}, \ref{prop: H2component}, and \ref{prop: HCMcomponent} imply that all the potential loci are contained in the closure of $\mathscr{H}_{2}$.
\end{proof}

\section{The space of stable pairs} \label{Space of stable pairs}
In this section, we give a full description of the irreducible components of the moduli space of PT stable pairs by summarizing the results in the previous section on the description of the associated moduli spaces to the chambers.

\begin{thm} \label{moduliPT1} Theorem \ref{moduliPT} holds.

\end{thm}

\begin{proof}
This is obtained from Propositions \ref{N1}, \ref{N2}, \ref{N4}, \ref{N5}, and \ref{N6} and Lemma \ref{survival}. We notice that all the loci appearing after each wall are new irreducible components as they have the maximal dimensions and cannot be considered as a subset of the previous components.

\end{proof}

Now, we summarize the birational description of the components of the intermediate moduli spaces:
\begin{thm} \label{moduliPTsecondtry}
Consider a path $\gamma\colon [0,1] \rightarrow\mathbb{R}^{+} \times \mathbb{R} \subset \Stab(\mathbb{P}^{3})$. The image of this path (outside of the walls) describes the moduli spaces of semistable objects with the fixed Chern character $(1,0,-6,15)$ for the corresponding chambers which are as follows:

  ($\mathcal{N}_{0}$): the empty space,
  
  ($\mathcal{N}_{1}$):   only contains the main component. 
  
  ($\mathcal{N}_{2}$):  the blow up of $\mathcal{N}_{1}$ along the smooth locus  $(\mathbb{P}^{3})^{*} \times (\mathbb{P}^{3})^{*} \times \mathcal{H}ilb(C_{2})$.

 ($\mathcal{N}_{3}$): has two irreducible components: $\widetilde{\mathcal{N}_{2}}$ and $\mathcal{N}'_{3}$. The component $\widetilde{\mathcal{N}_{2}}$ is obtained from $\NN_2$ by a small contraction followed by a divisorial extraction. The component $\mathcal{N}'_{3}$ is a new component described in (2) in Theorem \ref{moduliPT1}.

  ($\mathcal{N}_{4}$): has four irreducible components: $\widetilde{\widetilde{\mathcal{N}_{2}}}$, $\widetilde{\mathcal{N}'_{3}}$, $\mathcal{N}'_{4}$ and $\mathcal{N}''_{4}$.  
  
  The first two are birational to their counterparts  in  $\mathcal{N}_{3}$. The components   $\mathcal{N}'_{4}$  and $\mathcal{N}''_{4}$ are two new components described  in (3) and (4) in Theorem  \ref{moduliPT1}, respectively.

   ($\mathcal{N}_{5}$): has seven irreducible components: $\widetilde{\widetilde{\widetilde{\mathcal{N}_{2}}}}$, $\widetilde{\widetilde{ \mathcal{N}'_{3}}}$, $\widetilde{ \mathcal{N}'_{4}}$,   $\widetilde{ \mathcal{N}''_{4}}$,  $ \mathcal{N}'_{5}$, $\mathcal{N}''_{5}$, and $\mathcal{N}'''_5$.  
   The first four are birational to their counterparts  in  $\mathcal{N}_{4}$. The components   $\mathcal{N}'_{5}$,  $\mathcal{N}''_{5}$, and $\mathcal{N}'''_{5}$ are three new components described in (5), (6) and (7) in Theorem  \ref{moduliPT1}, respectively.

 ($\mathcal{N}_{6}$): has eight irreducible components: $\widetilde{\widetilde{\widetilde{\widetilde{\mathcal{N}_{2}}}}}$, $\widetilde {\widetilde{\widetilde{\mathcal{N}'_{3}}}}$, $\widetilde{\widetilde{\mathcal{N}'_{4}}}$, $\widetilde{\widetilde{\mathcal{N}''_{4}}}$, $\widetilde{\mathcal{N}'_{5}}$, $\widetilde{\mathcal{N}''_{5}}$,  $\widetilde{\mathcal{N}'''_{5}}$, and $\mathcal{N}'_{6}$.  
 
 The first seven are birational to their counterparts  in  $\mathcal{N}_{5}$. The component  $\mathcal{N}'_{6}$ is a new component described in (8)  in Theorem  \ref{moduliPT1}.

\end{thm}

\begin{proof}
This is obtained from Propositions \ref{N1}, \ref{N2}, \ref{N4}, \ref{N5}, and \ref{N6}, and Lemma \ref{survival}. The birational description of the wall between $\mathcal{N}_2$ and $\mathcal{N}_3$ is followed by \cite{R1}.

\end{proof}

\begin{rmk} \label{destabilizing}
From  Lemmas \ref{ext1green}, \ref{ext1purple1}, \ref{ext1purple2}, \ref{ext1purple3},  \ref{ext1pink}, the destabilizing loci for each of the moduli spaces corresponding to a chamber along the path on the right hand side in Figure \ref{walls} are described as follows:

  $(\mathcal{N}_{1}$):   There is an $11$ dimensional destabilizing locus $(\mathbb{P}^{3})^{*} \times (\mathbb{P}^{3})^{*} \times \mathcal{H}ilb(C_{2})$   in $\mathcal{N}_{1}$.
  
  ($\mathcal{N}_{2}$):  There is a destabilizing locus with two strata  in $\mathcal{N}_{2}$:   one is $\mathbb{G}r(2,4) \times \mathfrak{Fl}_2 $ of dimension 11, and the other is a $\P^1$-bundle over $\mathfrak{Fl}_2 $ of dimension 8.
  
 ($\mathcal{N}_{3}$):
  There is a destabilizing locus with two strata  in $\mathcal{N}_{3}$: one is $\mathbb{G}r(2,4)  \times \mathfrak{Fl}_1 $ of dimension 9, and the other is a $\P^1$-bundle over $(\P^2)^{*} \times \mathfrak{Fl}_1 $ of dimension 8.

   ($\mathcal{N}_{4}$): There is a destabilizing locus with two strata  in $\mathcal{N}_{4}$: one is  $\mathcal{U} \times (\mathbb{P}^{3})^{*}$ of dimension 8, and the other is a $\P^1$-bundle over $\mathfrak{Fl}_ 2$ of dimension 8.

   ($\mathcal{N}_{5}$): There is a destabilizing locus with three strata  in $\mathcal{N}_{5}$: one is a 14-dimensional locus parametrizing 6 points on a conic, the second locus is a $\P^1$-bundle over a 12-dimensional locus parametrizing 5 points on a line, and the third locus is a $\P^2$-bundle over an 11-dimensional locus parametrizing 6 points on a line.

\end{rmk}

\begin{rmk}
The exceptional locus in $\NN_{2}$ is a $\mathbb{P}^{12}$-bundle over  ${\mathcal{H}ilb^{2t+1}(\mathbb{P}^{3})} \times (\mathbb{P}^{3})^{\vee}$ which generically parametrizes union of a plane quartic and a conic intersecting in two points (see Proposition  \ref{N2} and the result in \cite{R1}).
\end{rmk}

\begin{exmp} We give examples of two objects in the space of PT stable pairs, such that when we move from the large volume limit to the empty set on the right side of the hyperbola, those objects get destabilized at the walls
   $\langle(\mathcal{O}_{\mathbb{P}^3}(-1) \rightarrow \mathcal{O}_{L}(1)), \mathcal{O}_{P}(-5)\rangle$
   and $\langle(\mathcal{O}_{\mathbb{P}^3}(-1) \rightarrow \mathcal{O}_{L}),\iota_{P_{*}}\mathcal{I}_{Z_{1}}^{\vee}(-5)\rangle$. 
   
   First we consider $E:=\mathcal{O}_{\mathbb{P}^3} \rightarrow \mathcal{O}_{L}(1) \oplus \mathcal{O}_{C_5}$ which is defined by a section of the object $\mathcal{O}_{L}(1) \oplus \mathcal{O}_{C_5}$, for $C_5$ a plane quintic.   We claim $\mathcal{O}_{\mathbb{P}^3}(-1) \rightarrow \mathcal{O}_{L}(1)$ destabilizes $E$ as follows (the right column comes from the section $\OO_{\mathbb{P}^3} \to \OO_L(2)$)
 \begin{center}
\begin{tikzcd}[column sep=
scriptsize
]
 \mathcal{O}_{\mathbb{P}^3}(-1) \arrow[rr, hookrightarrow] \arrow[d, ] && \mathcal{O}_{\mathbb{P}^3}\arrow[d, ] \arrow[rr, twoheadrightarrow] && \mathcal{O}_{P} \arrow[d, ]

\\
 \mathcal{O}_{L}(1) \arrow[rr, hookrightarrow]  &&  \mathcal{O}_{L}(1) \oplus \mathcal{O}_{C_5}\arrow[rr, twoheadrightarrow] &&   \mathcal{O}_{C_5} 
\end{tikzcd}
\end{center}
or 
$$(\mathcal{O}_{\mathbb{P}^3}(-1) \rightarrow \mathcal{O}_{L}(1)) \hookrightarrow E \twoheadrightarrow  \mathcal{O}_{P}(-5).$$
   Now we consider  $E' :=\mathcal{O}_{\mathbb{P}^3} \rightarrow \mathcal{O}_{L} \oplus \mathcal{O}_{C_5}(a)$ which is defined by a section of the object $ \mathcal{O}_{L} \oplus \mathcal{O}_{C_5}(a)$, for $C_5$ a plane quintic and $a \in C_5$.   We claim that  $\mathcal{O}_{\mathbb{P}^3}(-1) \rightarrow \mathcal{O}_{L}$ destabilizes $E'$ as follows (the right column comes from the section $\OO_{\mathbb{P}^3} \to \OO_L(1)$)
 \begin{center}
    
\begin{tikzcd}[column sep=
scriptsize
]
 \mathcal{O}_{\mathbb{P}^3}(-1) \arrow[rr, hookrightarrow] \arrow[d, ] && \mathcal{O}_{\mathbb{P}^3}\arrow[d, ] \arrow[rr, twoheadrightarrow] && \mathcal{O}_{P} \arrow[d, ]

\\
 \mathcal{O}_{L} \arrow[rr, hookrightarrow]  &&  \mathcal{O}_{L} \oplus \mathcal{O}_{C_5}(1)\arrow[rr, twoheadrightarrow] &&   \mathcal{O}_{C_5}(1)

\end{tikzcd}
\end{center}
or 
$$(\mathcal{O}_{\mathbb{P}^3}(-1) \rightarrow \mathcal{O}_{L}) \hookrightarrow E' \twoheadrightarrow  \iota_{P_{*}}\mathcal{I}_{Z_{1}}^{\vee}(-5),$$
as $\HH^0(\mathcal{O}_{P} \rightarrow \mathcal{O}_{C_5}(1))=\HH^0(\iota_{P_{*}}\mathcal{I}_{Z_{1}}^{\vee}(-5))= \mathcal{O}_{P}(-5)$ and on the other hand,  $\HH^1(\mathcal{O}_{P} \rightarrow \mathcal{O}_{C_5}(1))=\HH^1(\iota_{P_{*}}\mathcal{I}_{Z_{1}}^{\vee}(-5))= \mathcal{O}_{Z_1}$. 

\end{exmp}

\section{The Hilbert Scheme} \label{The Hilbert Scheme}
In this section, after describing the DT/PT wall and crossing it, we describe the Hilbert scheme ${\mathcal{H}ilb^{6t-3}(\mathbb{P}^{3})}$. We note that using results in \cite{JM2019},  we can interpret the DT/PT wall-crossing formula as a wall-crossing in the Bridgeland stability space (similar statements for different types of stability conditions can be found in {\cite[Section 6]{Bayer09polynomial}} and {\cite[Section 4.3]{Toda09limit}}).

\begin{prop}

 \label{genusbound}
Let $C \subset \mathbb{P}^{3}$ be a curve of degree $6$. Then we have $p_a(C) \leq 10$.
\end{prop}
\begin{proof}
 This is obtained from Castelnuovo inequality.
\end{proof}
Note that the constant term of the Hilbert polynomial for a curve is $1-p_a+n$, where $n$ is the number of floating or embedded points. Thus, as $n\geq 0$, for a Cohen-Macaulay curve $C$ of Hilbert polynomial $6t-3$  we must have $1-p_a(C)\leq 1-p_a+n=-3$, which means $4\leq p_a(C)$. Combining this with Proposition \ref{genusbound} implies that  
for a curve of Hilbert polynomial $6t-3$, we have $4 \leq p_a\leq 10$, and it has $p_a-4$ floating or embedded points.
\\\\
\noindent
\textbf{Description of the hyperbola as an actual wall.} Recall that the hyperbola is given by $\Imm (Z_{\alpha, \beta, s})=0$, and so  to describe the hyperbola as an actual wall, we need to find some objects $E$ such that $\Imm (Z_{\alpha, \beta, s}\mathrm{ch}(E))=0$. 

\begin{lem} \label{DT/PT}
As an actual wall, the hyperbola $\mathbb{H}$ is given by  $\langle\mathcal{I}_{C_{i}}, \mathcal{T}_{i}[-1]\rangle$ for $1 \leq i\leq 6$, \textcolor{red}{} where $C_{i}$ is a Cohen-Macaulay curve of degree 6 and  genus $4+i$,  and $\mathcal{T}_{i}$ is a torsion sheaf of length $i$.
\end{lem}

\begin{proof}
For objects with $\Imm (Z_{\alpha, \beta, s}\mathrm{ch}(E))=0$, semistability does not change as $s$ varies; in particular, we can let $s \to +\infty$ and apply Lemma \ref{BMS8.9} to $E[1]$. Therefore, $\HH^0_\beta(E)$ is $\nu_{\alpha, \beta}$-semistable, and $\HH^1_\beta(E)$ is a torsion sheaf $\TT_{i}$ of length $i$. Notice that we have $\ch(\HH^0_\beta(E))=(1,0,-6,15+i)$, which is the Chern class of the ideal sheaf of a genus 4+i sextic curve $C_{i}$. By Proposition \ref{genusbound}, the length cannot be more than 6. Notice that $C_i$ is a Cohen-Macaulay curve: Otherwise, there would be a 0-dimensional subsheaf $T$ of $\OO_{C_i}$, which induces an injection $T[-1] \into \II_{C_i}$. But $\II_{C_i}$ is stable; so this is a contradiction.

\end{proof}

Finally, we give a description of the Hilbert scheme and prove Theorem \ref{HScomp}:
\begin{thm} \label{HScomp1} Theorem \ref{HScomp} holds.

\end{thm}

\begin{proof}
Theorem  \ref{moduliPT1} describes the eight components of the space of stable pairs. We notice that crossing the DT/PT wall (described in Lemma \ref{DT/PT}) from the right side to the left side changes the heart from $\mathrm{Coh}^{ \alpha, \beta}(\mathbb{P}^{3})$ to $\mathrm{Coh}^{ \alpha, \beta}(\mathbb{P}^{3})[-1]$. The component $\mathscr{H}_{CM}$ (which is birational to the main component) and the component $\mathscr{H}'_{CM}$  (which is birational to $\mathcal{N}'''_{5}$) survive after crossing this wall: we just need to check this for one object in each; let $\II$ be either an ideal sheaf of a (2,3)-complete intersections or an ideal sheaf of the union of a plane quartic with a thickening of a line in the plane (see Propositions \ref{N1} and \ref{N5}). Let $C_i$ be a Cohen-Macaulay curve of degree $6$ and genus $4+i$, where $1\leq i \leq 6$. As neither a (2,3)-complete intersection curve nor the union of a plane quartic with a thickening of a line in the plane can be contained in a $C_i$, we have $\mathrm{Hom}(\II_{C_i},\II)=0$. Therefore, $\II$ cannot get destabilized at this wall. 

We notice that apart from these two components of genus 4 Cohen-Macaulay curves, other six components in the space of stable pairs are destabilized once we cross the DT/PT wall: each object $E$ in those components has $\HH^1(E) \neq 0$ (see Propositions \ref{N3}, \ref{N4}, \ref{N5} and \ref{N6}), and the short exact sequence $\HH^0(E) \rightarrow E \rightarrow \HH^1(E)[-1] $  destabilizes any object in those components. 

Note that the underlying Cohen-Macaulay curves which are parametrized by the generic elements of the components remain the same in both the Hilbert scheme and the space of stable pairs (as the support of $\FF$, for a given stable pair $(\FF,s)$). Therefore, the new components in the Hilbert scheme generically parametrize ideal sheaves of curves with underlying Cohen-Macaulay curves either the disjoint union  of a line in $\P^3$ and a plane quintic, or the union of a line in $\P^3$ and a plane quintic intersecting the line, or a plane sextic. They have 1,2, or 6 floating/embedded points, respectively (see and Propositions \ref{prop: H1component}, \ref{prop: H2component}, \ref{prop: HCMcomponent}, \ref{N6}). The dimension of $\mathscr{H}_1$ is obtained from Proposition \ref{prop: H1component}.

As no object in the components $\mathscr{H}_{CM}$ and $\mathscr{H}'_{CM}$ gets destabilized after crossing the DT/PT wall, the dimensions and irreducibility of  them are clear from the stable pairs side. For $\mathscr{H}_{1}$, there is only one point involved and the underlying curve is a disjoint union of a line and a plane quintic which is a locally complete intersection curve of codimension 2; hence {\cite[Theorem 1.3]{CN12}} implies that we can always detach the point from the underlying curve in this case. Therefore $\mathscr{H}_{1}$ is irreducible.

As for $\mathscr{H}_{2}$, Corollary \ref{Corollary: irred} implies the irreducibility.
\end{proof}
Notice that our wall-crossing picture confirms the classical result below by Hartshorne for canonical genus four curves, as the only non-plane curves that occur as underlying curves of independent components have arithmetic genera $4,5,6$:

\begin{thm} [{\cite[Theorem 3.3]{Harts}}] \label{Hartshormethm}
Let $C$ be a locally Cohen-Macaulay curve in $\mathbb{P}^{3}$, of degree
$d \geq 3$, which is not contained in a plane. Then
$$g(C)\leq \frac{(d-2)(d-3)}{2}.$$
\end{thm}

\bibliographystyle{siam}
\bibliography{sample}
\end{document}